[1994/12/01]
\documentclass[10pt, reqno]{amsart}
\usepackage[english, activeacute]{babel}
\usepackage{amsmath,amsthm,amsxtra}
\usepackage{epsfig}
\usepackage{amssymb}
\usepackage{latexsym}
\usepackage{amsfonts}
\usepackage{hyperref}
\usepackage{multicol,color}

\title{Dynamics of breathers in the Gardner hierarchy: universality of the variational characterization} 

\author{Miguel A. Alejo}
\thanks{M. A. was partially funded by Product. CNPq grant no. 305205/2016-1, IMUS and VI PPIT-US program ref. I3C}
\author{Eleomar Cardoso}
\address{Departamento de Matem\'atica, Universidade Federal de Santa Catarina, Brasil}
\email{miguel.alejo@ufsc.br}

\email{eleomar.junior@ufsc.br}
\date{\today}
\thanks{}

\chardef\bslash=`\\ 

\newtheorem{thm}{Theorem}[section]
\newtheorem{cor}[thm]{Corollary}
\newtheorem{lem}[thm]{Lemma}
\newtheorem{prop}[thm]{Proposition}

\newtheorem{defn}[thm]{Definition}

\theoremstyle{remark}
\newtheorem{rem}{Remark}[section]

\numberwithin{equation}{section}

\newcommand{\R}{\mathbb{R}}
\newcommand{\N}{\mathbb{N}}
\newcommand{\Z}{\mathbb{Z}}
\newcommand{\T}{\mathbb{T}}

\newcommand{\al}{\alpha}
\newcommand{\bt}{\beta}
\newcommand{\ga}{\gamma}

\newcommand{\px}{\partial_x}

\newcommand{\sech}{\operatorname{sech}}

\newcommand{\re}{\operatorname{Re}}
\newcommand{\ima}{\operatorname{Im}}

\newcommand{\be}{\begin{equation}}
\newcommand{\ee}{\end{equation}}
\newcommand{\bp}{\begin{proof}}
\newcommand{\ep}{\end{proof}}
\newcommand{\bel}{\begin{equation}\label}
\newcommand{\eeq}{\end{equation}}
\newcommand{\bea}{\begin{eqnarray}}
\newcommand{\eea}{\end{eqnarray}}
\newcommand{\bee}{\begin{eqnarray*}}
\newcommand{\eee}{\end{eqnarray*}}
\newcommand{\ben}{\begin{enumerate}}
\newcommand{\een}{\end{enumerate}}
\newcommand{\nonu}{\nonumber}

\newcommand{\ms}{\medskip}

\newcommand{\sn}{\operatorname{sn}}
\newcommand{\nd}{\operatorname{nd}}

\newcommand{\eval}[2][\right]{\relax
  \ifx#1\right\relax \left.\fi#2#1\rvert}

\begin{document}
\begin{abstract}
We present a new variational characterization of  breather solutions of any equation of the \emph{focusing}
Gardner hierarchy. This hierarchy is characterized by a nonnegative index $n$, and $2n+1$ represents the order of the corresponding PDE member. 
In this paper, we first show the existence of such breathers, and that they are solutions of the (2n+1)th-order Gardner equation.
Then we prove a \emph{variational universality property}, in the sense that all these breather solutions satisfy the \emph{same}
fourth order stationary elliptic ODE, regardless the order of the hierarchy member. This fact also characterizes them as critical points of the same Lyapunov functional, that
we also construct here. As by product of our approach, we find breather solutions of the hierarchy of
(2n+1)th-order mKdV equations, as well as a respective characterization of them as solutions of a fourth order stationary elliptic ODE. 
We also extend part of these results to the periodic setting, presenting new breather solutions for the 5th and 7th mKdV members of the hierarchy.
Finally, we prove ill-posedness results for the whole Gardner hierarchy, by using appropiately their breather solutions.
\end{abstract}
\maketitle \markboth{ Breathers of higher order Gardner equations}{Miguel A. Alejo and Eleomar Cardoso}
\renewcommand{\sectionmark}[1]{}

\section{Introduction}\label{0}
In this work we are concerned with the \emph{focusing} Gardner hierarchy  that we define as follows

\be\begin{aligned}\label{nthfocG}
& u_t = - \frac{\partial}{\partial x}\Big(-i\frac{\partial}{\partial x} + 2(\mu+u)\Big)\mathcal L_{n}[iu_x + (\mu+u)^2],\\
&\\
&\quad \quad u(t,x)\in\R,\quad\quad~~\mu\in\R^+,\quad\quad~n\in\N,\\
\end{aligned}\ee

\medskip
\noindent
which we will call it hereafter as $(2n+1)$th-order \emph{focusing} Gardner equations. Here $\mathcal L_n$ are the \emph{Lenard} operators defined  recursively by

\be\begin{aligned}\label{lenard}
&\frac{\partial}{\partial x}\mathcal L_{n+1}[v] = \Big(\frac{\partial^3}{\partial x^3}+4v\frac{\partial}{\partial x} + 2v_x\Big)\mathcal L_{n}[v] ,\quad n\in\N.
\end{aligned}\ee

For instance, starting with $\mathcal L_{0}[v] = \frac{1}{2}$, we get

\be\begin{aligned}\label{lenard2a}
\mathcal L_{1}[v] =&~{}v,\\
\mathcal L_{2}[v] =&~{} v_{xx} + 3v^2,\\
\mathcal L_{3}[v] =&~{} v_{4x} + 10vv_{xx}+5v_x^2+10v^3,
\end{aligned}\ee
and
\be\begin{aligned}\label{lenard2b}
\mathcal L_{4}[v] =&~{}v_{6x} +14vv_{4x} +28v_xv_{3x} + 21v_{xx}^2 + 70v^2v_{xx} + 70vv_{x}^2 + 35v^4,\\
\mathcal L_{5}[v] =&~{} v_{8x} + 18vv_{6x} + 54v_xv_{5x} +114v_{2x}v_{4x} + 126v^2v_{4x} + 69v_{3x}^2  \\
&+ 504vv_{x}v_{3x} + 420v^3v_{xx} + 390vv_{2x}^2 + 630v^2v_{x}^2 \\
& + 462v_x^2v_{2x} + 966v_x^2v_{3x} + 1260v^5v_x^2 + 126v^5.
\end{aligned}\ee

%

This recursion relation \eqref{nthfocG} generates the whole \emph{focusing} Gardner hierarchy. Note that we will only work with the focusing version
(via the complex Miura transformation $iu_x + (\mu+u)^2$ \cite{Miura})
since we are
interested in real valued and regular solutions (namely, solitons and breathers), instead of the singular structures which appear in the defocusing Gardner hierarchy,
and which, on the other hand, was already studied by Gomes et al \cite{gomes}. Moreover, only  positive values for the parameter $\mu$ will be considered in the present work,
since as it was recently discovered by Mu\~noz and Ponce \cite{mupo}, Gardner breathers do not exist when  $\mu\in\R^{-}$. See below for more details of 
breather solutions \eqref{nthbreatherG}.

\medskip

 Using the Lenard operators $\mathcal L_n$ \eqref{lenard}-\eqref{lenard2a}-\eqref{lenard2b} and a suitable spatial translation to remove linear terms, 
 a first few members of the  Gardner hierarchy are  (see appendix \ref{9mkdv} for 9th and 11th order Gardner equations):

\medskip

\noindent
the \emph{focusing} Gardner ($n=1$)
\be\begin{aligned}\label{examplesG}
&u_{t} +(u_{xx}  +6\mu u^2 + 2u^3)_x=0,\\
\end{aligned}\ee

\noindent
\emph{the 5th-order Gardner} ($n=2$)
\be\begin{aligned}\label{examplesG5}
& u_t + u_{5x}  +10 \mu ^2 u_{3x}+20 \mu  u u_{3x}+10 u^2 u_{3x}+120 \mu^3 u u_x\\
&+180 \mu ^2 u^2 u_x+120 \mu  u^3 u_x+10 u_x^3+30 u^4u_x+40 \mu  u_xu_{xx}+40uu_xu_{xx}=0,\\
\end{aligned}\ee

\medskip
\noindent
\emph{the 7h-order Gardner} ($n=3$)
\be\begin{aligned}\label{examplesG7}
&u_{t} + u_{7x} +14 \mu ^2u_{5x}+28 \mu uu_{5x}+14u^2u_{5x}+70 \mu ^4 u_{3x}+280 \mu ^3 uu_{3x}\\
&+420 \mu ^2 u^2u_{3x}+280 \mu u^3u_{3x}+70u^4u_{3x} +840 \mu ^5uu_x+2100 \mu ^4u^2u_x+2800 \mu ^3 u^3u_x\\
&+420 \mu ^2u_x^3 +2100 \mu ^2 u^4u_x+840 \mu uu_x^3 +840 \mu u^5u_x +420 u^2u_x^3+140 u^6u_x\\
&+84 \mu u_{4x}u_x+84 uu_xu_{4x} +140 \mu u_{xx}u_{3x} +140 uu_{xx}u_{3x}+126 u_x^2u_{3x}+560 \mu ^3 u_xu_{xx}\\
&+1680 \mu ^2 uu_xu_{xx}+1680 \mu u^2u_xu_{xx}+182 u_xu_{xx}^2 +560 u^3u_xu_{xx}=0.
\end{aligned}\ee
%
%
Note, moreover that selecting $\mu=0$ in
\eqref{nthfocG}, we get the \emph{focusing} mKdV hierarchy of equations, defined by

\be\begin{aligned}\label{nthfocmkdv}
& u_{t} = - \frac{\partial}{\partial x}\Big(-i\frac{\partial}{\partial x} + 2u\Big)\mathcal L_{n}[iu_x + u^2],\quad n\in\N.\\
\end{aligned}\ee
\noindent

This recursion relation \eqref{nthfocmkdv} generates the whole mKdV hierarchy, recovering the one depicted in \cite{olver}-\cite{Mat1}. Many previous works have shown 
explicit solutions of hierarchies. For the mKdV hierarchy,  Matsuno \cite{Mat1} proved the existence and built 
explicitly the $N$-soliton solution. More recently, Gomes et al \cite{gomes} dealt with the
defocusing mKdV with non vanishing boundary conditions (NVBC) and the associated defocusing Gardner hierarchy, showing multisolitonic structures. Unfortunately, many of the zero boundary 
value solutions are singular.

\medskip

With respect to breather solutions, they are defined as localized in space and perodic in time 
(up to symmetries of the equation) functions. For instance, in the Gardner case, they are \emph{shortly} defined as follows:

\begin{defn}[Gardner breather]\label{5BreG}
Let $\al, \bt \in \R\backslash\{0\},~\mu\in \R^{+}\backslash\{0\}$ such that $\Delta=\al^2+\bt^2-4\mu^2>0$. A  breather solution $B_\mu$ 
of the \emph{classical} Gardner equation \eqref{examplesG} is given  by the formula
\be\label{breatherintro}B_\mu\equiv B_{\al, \bt,\mu}(t,x;x_1,x_2)   
 :=   2\partial_x\Bigg[\arctan\Big(\frac{\tfrac{\bt\sqrt{\al^2+\bt^2}}{\al\sqrt{\Delta}}\sin(\al y_1) - h_1(t,x)}{\cosh(\bt y_2)-h_2(t,x)}\Big)\Bigg],
\ee
where $h_1,h_2$ are precise trigonometric and hyperbolic functions.
\end{defn}

See Theorem \ref{nthfocG} for a detailed and complete definition. Breathers were previously studied  for  mKdV \cite{AM,AM1},
Gardner \cite{AM11}, sine-Gordon \cite{AMP1}  and NLS \cite{AMF} equations and also in many other nonlinear models (see 
\cite{KKS1,KKS2,jesus1,jesus2,jesus3}). In some of these works, breather solutions are 
indeed characterized as solutions of a precise fourth order ODE (mKdV and Gardner) or a system of ODEs (sine-Gordon), and indeed 
being defined as local minimizers of suitable Lyapunov functionals built as linear combinations of conserved quantities up to $H^2$ level. 
In the case of higher order equations, some stability results were proved for breather solutions  of the 5th, 7th and 9th mKdV equations in  \cite{AleCar} and also for the breather solution of the 5th Gardner equation, once the global in time behavior 
of its solutions was well understood \cite{AleCk,AleCar2}. However, with respect to Gardner or mKdV hierarchies and as far as we know, no real regular breather 
solutions were shown or mentioned explicitly in the literature \cite{Mat,gomes}.

%
%


\medskip
In this work we present breather solutions for the whole Gardner and mKdV hierarchies \eqref{nthfocG}-\eqref{nthfocmkdv}, and we will show that in fact they share the 
same functional profile for the whole hierarchy, up to corresponding speed parameters, 
which will depend on the level of the hierarchy considered (see \eqref{nthbreatherG} for a detailed definition). This is for us a nonexpected and surprising \emph{universality property}. Even more, in the periodic mKdV setting, we obtain 
a detailed description of \emph{periodic} breather solutions of the corresponding 5th and 7th order mKdV equations (see Section \ref{periodicB}). We believe in 
fact that a similar complete description of periodic breather solutions for the whole mKdV and Gardner hierarchies is feasible.

\medskip

Also the main aim of this work is to show that all breather solutions of Gardner and mKdV hierarchies satisfy a \emph{universal} fourth order ODE, which it is the same for
any breather solution of the corresponding equation member of the considered hierarchy. This universality expands the variational characterization of these 
breather solutions, meaning that for any member of the Gardner and mKdV hierarchies, a breather solution is a critical point of a precise Lyapunov functional 
defined in the Sobolev space $H^2$.

\medskip

Therefore, the goal of this work is twofold: firstly we are going to define and to characterize variationally regular breather and N-soliton solutions of the  Gardner
and mKdV  hierarchies of equations. Secondly, the most important result will be to show that
any higher order Gardner (mKdV) breather solution of the Gardner (mKdV) hierarchy of equations \eqref{nthfocG}-\eqref{nthfocmkdv},
satisfies the \emph{same} fourth-order stationary elliptic ODE and it is a critical point of a Lyapunov functional defined in $H^2$. That means that the  breather solution of the corresponding equation member of the
Gardner (mKdV) hierarchy holds a fourth order ODE, which it is the same for all breather solutions of any  higher order Gardner (mKdV) equation of the
Gardner  (mKdV) hierarchy, and therefore the same universal ODE independently of the level of the Gardner (mKdV) hierarchy considered. 
Using these higher order breather solutions of the Gardner hierarchy, we will prove the ill-posedness of the Gardner hierarchy, determining 
the critical Sobolev index depending on the level of the hierarchy. Finally, and for the shake of completeness, we will provide a complete description of
periodic breather solutions of 5th and 7th order mKdV equations.

\medskip

In short, we list our main results as follows:


\begin{thm}[Existence of Gardner hierarchy breathers]\label{MTG0}  For all $n\in\N$ there exists a breather solution of the corresponding $(2n+1)$-th order Gardner equation of the Gardner hierarchy
 \eqref{nthfocG}. More precisely
 
 \begin{enumerate}
  \item \emph{~Structure}:~~ all these breather solutions have the same functional structure, namely given 
 $\al, \bt\in \R\backslash\{0\}$ and $\mu \in \R^+\backslash\{0\}$ 
 such that $\Delta=\al^2+\bt^2-4\mu^2>0$,  and $x_1,x_2\in \R$ we have
 \be\label{nthbreatherG}
  B_\mu:=
  2\partial_x\left[\arctan\left(
  \frac{\tfrac{\bt\sqrt{\al^2+\bt^2}}{\al\sqrt{\Delta}}\sin(\al y_1) -\tfrac{2\mu\bt[\cosh(\bt y_2)+\sinh(\bt y_2)]}{\Delta}}
  {\cosh(\bt y_2)-\tfrac{2\mu\bt[\al\cos(\al y_1)-\bt\sin(\al y_1)]}{\al\sqrt{\al^2+\bt^2}\sqrt{\Delta}}}\right)\right],
 \ee
\noindent with $y_1 = x+ \delta_{2n+1,\mu} t + x_1$ and $y_2 = x+ \ga_{2n+1,\mu} t +x_2$.
\medskip
  \item \emph{~Velocities}:~~ The velocities $\ga_{2n+1,\mu}$ and $\delta_{2n+1,\mu}$ are the only parameters depending on the level of the hierarchy considered, in the following explicit form:
 \[\begin{aligned}
&\ga_{2n+1,\mu} := -\frac{1}{\bt} \re\Big[\sum_{p=1}^{n}a_{p,n}(\bt + i\al)^{2p+1}\mu^{2(n-p)}\Big],\\
&\delta_{2n+1,\mu} := -\frac{1}{\al} \ima\Big[\sum_{p=1}^{n}a_{p,n}(\bt + i\al)^{2p+1}\mu^{2(n-p)}\Big],\\
\end{aligned}\]
\noindent%
where
\[\begin{aligned}
&a_{p,n}~\equiv~\text{coefficient of the term}~~ \mu^{2(n-p)}u_{(2p+1)x} \\
&~\text{in  the (2n+1)th-order Gardner equation}~~ \eqref{nthfocG}.
\end{aligned}\]
  \item \emph{~ Smoothness}:~~ Each breather solution of the Gardner hierarchy above presented is smooth in time and space, 
  and belong to the Schwartz class in space.
  \smallskip
  \item \emph{~Convergence to mKdV}: For all $n\in\N$, these breather solutions of the Gardner hierarchy reduce to 
 breather solutions of the mKdV hierarchy \eqref{nthfocmkdv} as $\mu\to0$, namely
 \be\label{intronthBresmkdv}B\equiv B_{\al, \bt,n}(t,x;x_1,x_2)   
 :=   2\partial_x\Bigg[\arctan\Big(\frac{\bt}{\al}\frac{\sin(\al y_1)}{\cosh(\bt y_2)}\Big)\Bigg],
\ee
\noindent
with $y_1$ and $y_2$
\be\label{introy1y2mkdv}
\begin{aligned}
&y_1 = x+ \delta_{2n+1} t + x_1, \qquad y_2 = x+ \ga_{2n+1} t +x_2,~~~\end{aligned}\ee
\noindent and with velocities
\[
\delta_{2n+1} := -\frac{1}{\al} \ima\Big[(\bt + i\al)^{2n+1}\Big],\qquad  \ga_{2n+1} := -\frac{1}{\bt} \re\Big[(\bt + i\al)^{2n+1}\Big].
\]
 \end{enumerate}
\end{thm}

\begin{rem}
As far as we know, this is the first example of breather solutions of the mKdV and Gardner hierarchies.
\end{rem}
%

Our second result is a characterization of each new Gardner breather as solution of a nonlinear fourth order ODE. Note that this ODE does not depend on the 
order of the hierarchy. Furthermore, any  breather solution  of the hierarchy is a critical point of the same Lyapunov functional. 

\begin{thm}[Universal nonlinear ODE and variational characterization]\label{introMTG0}  Any breather solution $B_\mu$  of the $(2n+1)$-th order 
Gardner equation \eqref{nthfocG} satisfies, for any $n\in\N$, the \emph{same} fourth order elliptic equation
\be\label{introMTG0a}
\begin{aligned}
& B_{\mu,4x} + 2 (\al^2-\bt^2)(B_{\mu,xx} + 6\mu B_\mu^2 + 2B_\mu^3) + (\al^2+\bt^2)^2 B_\mu  \\
& \quad + 10B_\mu^2 B_{\mu,xx} + 10B_\mu B_{\mu,x}^2 + 6B_\mu^5 + 10\mu B_{\mu,x}^2 \\ 
& \quad + 20\mu B_\mu B_{\mu,xx} + 40\mu^2B_\mu^3 + 30\mu B_\mu^4 =0.
\end{aligned}
\ee
\noindent
Moreover, for any $n\in\N$, the breather solution $B_\mu$ of the Gardner hierarchy \eqref{nthfocG} is a critical point of a universal Lyapunov functional
$\mathcal H_\mu$ \eqref{LyapunovGE}, which 
is written as linear combination of three conserved laws in the following way
\be\label{Hmu}
 \mathcal H_\mu[u(t)] := F_\mu[u](t) + 2(\bt^2-\al^2)E_\mu[u](t) + (\al^2 +\bt^2)^2 M[u](t).
\ee
\end{thm}


\begin{rem}
The conserved functionals in \eqref{Hmu} can be explicitly found in \eqref{Mass}, \eqref{E1N} and \eqref{E2}. 
\end{rem}

\begin{rem}
This nonlinear ODE was already found in the case of the classical Gardner equation in \cite{AM11}. See also \cite{AM} for a first proof in the mKdV case. 
The surprise here is that every breather member of the hierarchy is solution of the same nonlinear ODE.
\end{rem}

\begin{rem}[About the stability of the hierarchy Gardner breathers]
Note that with this variational characterization for breather solutions of the whole Gardner hierarchy, and applying the same ideas pointed out in 
\cite{AM11} (see also \cite{AM}), a suitable {\bf stability result} for breather solutions of the Gardner hierarchy \eqref{nthfocG} can be presented, provided a well-posedness theory is 
available, \emph{which is not the case today}. Currently, only the stability result for breather solutions of the 5th order Gardner equation has been proved, see \cite{AleCk}. 
The well-posedness of the remaining members of the hierarchy is an interesting open problem. Compare with Theorem \ref{ill2n1th0} below, which shows weak ill-posedness depending on the index $n$ of the member. 
\end{rem}

\medskip


%
%
%

As a direct corollary of the main result Theorem  \ref{MTG0}, we  get, when $\mu=0$ in \eqref{introMTG0a}, the fourth-order  elliptic ODE
satisfied by all breather solutions of the mKdV hierarchy.

\begin{cor}[Universality for mKdV hierarchy breathers]\label{aa4thodemkdv0}  Any breather solution $B$ of the $(2n+1)$-th order mKdV equation \eqref{nthfocmkdv}
satisfies for any $n\in\N$ the \emph{same} fourth order  elliptic equation
\be\label{aa4thodemkdv0a}
\begin{aligned}
& B_{4x} + 2 (\al^2-\bt^2)(B_{xx} + 2B^3) + (\al^2+\bt^2)^2 B \\
&\qquad + 10B^2B_{xx}+ 10BB_x^2 + 6B^5 =0.
\end{aligned}
\ee
Moreover, mKdV breathers of the hierarchy are classical critical point of an associated functional, just as in \ref{Hmu}, but with $\mu=0.$
\end{cor}

This corollary complements and completes the main result in \cite[Th.2.4]{AleCar}, which essentially considered the 5th, 7th and 9th mKdV equations only. 

\medskip

Our third result is related to well-posedness issues. Indeed, following the key idea introduced by Kenig, Ponce and Vega \cite{KPV2} on NLS and mKdV, and generalized 
to the classical Gardner equation in \cite{Ale1} and the 5th order Gardner equation in \cite{AleCar2}, we are able to study these higher order breather solutions, and show ill-posedness for the whole hierarchy in the following sense: 

\begin{thm}[Growing ill-posedness for the Gardner hierarchy]\label{ill2n1th0}
If $s<\frac{2n-1}{4}$, the mapping data-solution $u_0\rightarrow u(t)$, with $u(t)$ a solution of the IVP for the
(2n+1)th-order Gardner equation \eqref{nthfocG} is not uniformly continuous.
\end{thm}

\begin{rem}
For the sake of completeness, we write explicitly the critical threshold index $s_c(n)$ for well-posedness, according to Theorem \ref{ill2n1th0}: for $n=1$, $s_c(1)=\frac14$ (classical Gardner), for $n=2$, $s_c(2)=\frac34$ (5th. order, see \cite{AleCk}), for $n=3$, $s_c(3)=\frac54>1$  (7th. order, above standard energy space $H^1$), for $n=4$, $s_c(4)=\frac74<2$ (9th order), and  $s_c(5)=\frac94>2$ (11th order, above $H^2$ energy space), and so on.
\end{rem}

Note that the larger is $n$, the higher is the Sobolev index $s$ for which we have ill-posedness below that regularity. 
Additionally, for $n>4$ we have ill-posedness above $H^2$, the natural space for breathers stability. 
Consequently, we cannot expect variational stability of breathers for members of the hierarchy of order $2n+1=11$ or higher.

\medskip

Also, a completely similar ill-posedness result can be proved for the mKdV hierarchy. See Corollary \ref{ill2n1thmkdv}. The Sobolev index inequality $s<\frac{2n-1}{4}$ is the same. 

\medskip

Finally, our last result considers the mKdV hierarchy in the periodic setting. 
We have obtained \emph{periodic}  breather solutions of higher order mKdV equations, as follows:

\begin{thm}[Existence of higher order periodic breathers]\label{introperiodicmkdv}
The 5th and 7th order mKdV equations have \emph{periodic} breather solutions, of the following form:
\be\label{intro57Bper}
B= B(t,x; \al, \bt, k, m, x_1,x_2)   :=  
2\partial_x \Big[ \arctan \Big( \frac{\bt}{\al}\frac{\sn(\al y_1,k)}{\nd(\bt y_2,m)}\Big) \Big],
\ee
where $\sn(\cdot ,k)$ and $\nd (\cdot , m)$ are the standard Jacobi elliptic functions of elliptic modulus $k$ and $m$, and
\be\label{introY1_Y2}
y_1:=x+\delta_{i,m} t + x_1,  \quad y_2:=x+\ga_{i,m} t +x_2,\quad i=5,7.
\ee
\end{thm}

\begin{rem}
See \eqref{perspeeds5} and \eqref{perspeeds7} in Section \ref{periodicB} for an explicit expression 
for velocities $(\delta_{5,m},\ga_{5,m})$ in the 5th order case and  $(\delta_{7,m},\ga_{7,m})$ in the 7th order respectively.
\end{rem}

\begin{rem}
See \cite{AMP2} for a detailed account on the elliptic functions involved in \eqref{intro57Bper}.
\end{rem}

The validity Theorem \ref{introperiodicmkdv} follows directly, after cumbersome computations, or the use of a standard symbolic software. We skip the details for the interested reader.

\medskip

We believe that these periodic breathers are completely new for the 5th and 7th order setting. 
In the classical mKdV periodic case, these solutions were found by Kevrekidis et al. \cite{KKS1,KKS2}. 
Also, in \cite{AMP2}, we presented stability properties of these solutions, which could be applied to these new breathers after some work. 
The variational structure of these solutions is an interesting open problem.

\medskip

We believe that the whole mKdV hierarchy has the same functional expression for periodic breather solutions, 
varying with velocities, but the form of higher order speeds than 5th and 7th order has escaped to us, and 
we were not able yet to obtain a complete description.

\subsection{Organization of this paper}
This paper is organized as follows: in Section \ref{1} we introduce the solitons for the Gardner hierarchy. In Section \ref{2} we prove existence of breathers, Theorem \ref{MTG0}. Section \ref{Sect:4} deals with the proof of the variational characterizations of Gardner breathers, Theorem \ref{introMTG0}. Section \ref{ill} is devoted to the proof of Ill-posedness of the Gardner and mKdV hierarchies, Theorem \ref{ill2n1th0}. Finally, Section \ref{periodicB} provides further information on Theorem \ref{introperiodicmkdv}.    

 \bigskip

{\bf Acknowledgments.}
We would like to thank to professors C. Mu\~noz and C. Kwak for useful discussions and comments. We are also indebted to
professor Y. Matsuno for his inspirating works and enlighting comments in a first version of this manuscript.

\section{Preliminaries}\label{1}

\subsection{Formulae for 1-solitons} As a consequence of the complete integrability of \eqref{nthfocG}, and following Matsuno
\cite[(3.11)]{Mat1}-\cite{Mat2} and the Inverse Scattering Method, it is possible to see that the  Gardner hierarchy \eqref{nthfocG}
has explicit 1-soliton solutions. For the sake of completeness, since we have not found a formal statement of this result in the literature, we include  it  here:

\begin{defn}\label{nthsoltionG}
The  higher order 1-soliton solution $Q_c\equiv Q_{c,n}$  of any equation member of the  Gardner hierarchy  \eqref{nthfocG},
i.e. of any $(2n+1)$-th order Gardner equation is given by

\be\begin{aligned}\label{examplesolGs} Q_{\mu}(t,x) := &~{} Q_{\mu,c} (x-v_{\mu,c}t),\\
Q_{\mu,c} (z) :=&~{} \frac{c^2}{2\mu + \sqrt{4\mu^2+c^2}\cosh(cz)},\qquad c>0,
\end{aligned}
\ee
with
\be\begin{aligned}
&v_{\mu,c}:=\sum_{p=1}^{n}a_{p,n}c^{2p}\mu^{2(n-p)},\qquad \forall n\in\N,\nonu
\end{aligned}\ee
and $a_{p,n}$ as in the Definition \ref{nthbreatherG} above.
\end{defn}
%



Moreover, it is easy to see that any 1-soliton solution $Q_{\mu,c}$ \eqref{examplesolGs} of any equation member
of the Gardner hierarchy \eqref{nthfocG} satisfies the \emph{same} nonlinear stationary elliptic equation:

\medskip

\begin{prop}
Any soliton solution $Q_{\mu,c}\equiv Q_{c,n}$ \eqref{examplesolGs} of the $(2n+1)$-th order  Gardner equation 
\eqref{nthfocG} satisfy the \emph{same} second order stationary elliptic equation, independently of the index $n\in\N$,


\be\label{eqQc}\begin{aligned}
& Q_\mu'' -c\, Q_\mu + 6\mu Q_\mu^2 + 2Q_\mu^3=0, \\
& Q_\mu>0, \quad Q_{\mu}\in H^1(\R).
\end{aligned}\ee
\end{prop}

\begin{proof}
 By substituting directly  $Q_\mu$ \eqref{examplesolGs} in \eqref{eqQc} for all $n\in\N$.
\end{proof}

\medskip
\noindent

 Finally note that selecting the parameter $\mu=0$ in \eqref{examplesolGs} we reduce to the mKdV limit, and we get from the 1-soliton solution of the Gardner hierarchy \eqref{nthsoltionG}, the 1-soliton solution of the  mKdV hierarchy, as it was depicted by Matsuno \cite{Mat1}:


\begin{defn}[1-soliton solution of the  mKdV hierarchy]\label{1solitonmkdv}
The 1-soliton solution of the (2n+1)th-order mKdV equation \eqref{nthfocmkdv} is given by
\be\label{eqQcmkdv}\begin{aligned}
Q_{\mu}(t,x) := &~{} Q_{\mu,c} (x-v_nt),\\
Q_{\mu,c} (z) :=&~{} \sech(cz),\qquad c>0,
\end{aligned}
\ee
with
\be\begin{aligned}
&v_{n}:=c^{2n},\qquad \forall n\in\N.\nonu
\end{aligned}\ee

\end{defn}

\medskip

In addition to these higher order 1-soliton solution \eqref{examplesolGs}, we are also able to obtain \emph{breather} solutions of the Gardner hierarchy \eqref{nthfocG}.

\medskip

\section{Higher order breathers. Proof of Theorem \ref{MTG0}}\label{2}

Note that item (4) in Theorem \ref{MTG0} follows directly from proving items (1), (2) and (3). On the other hand, item (3) is a direct consequence of item (1). Given the speeds in item (2), we are only left to prove item (1) in Theorem \ref{MTG0}.

%
\medskip

Our first step will be to present the following identity valid for any solution of the  Gardner hierarchy \eqref{nthfocG}, and that will be a key tool in the proof of item 
(1) in Theorem \ref{MTG0}:

\begin{lem}\label{u2hier}
 Let $u(t,x)=i\partial_x\log(\frac{F_\mu-iG_\mu}{F_\mu+iG_\mu})=2\partial_x\arctan(\tfrac{G_\mu}{F_\mu})$ be any solution of the  Gardner hierarchy \eqref{nthfocG} for any $t\in\R$. 
 Then $u$ satisfies
 
 \be\label{B2}
 u^2 -  \frac{\partial^2}{\partial x^2}\log(F_\mu^2 + G_\mu^2) + 2\mu u =0.
 \ee
 
\end{lem}

\begin{proof}

We select a solution $u$ of \eqref{nthfocG} in the form
\begin{align}\label{ansatz1}
&~{} u(t,x):=  \phi_x,\qquad~\phi(t,x):=i\log \left(\frac{{\bf G}(t,x)}{{\bf F}(t,x)} \right),\\
&\text{where} \nonu\\\
 &~{} {\bf F} := F_\mu + iG_\mu,~~~ {\bf G} = F_\mu -iG_\mu = {\bf F}^{*}.
\end{align}
\noindent
First note  that here $F_\mu$ and $G_\mu$ are not necessarily 
the same functions introduced in \eqref{nthbreatherG} but generic ones for this ansatz. Note moreover, that the left part of the Gardner hierachy \eqref{nthfocG} can be rewriten
as a finite sum of terms like $\prod_{l=0}^{L}(u_{lx})^{d_l}$, with $u_{lx}=\partial^lu/\partial x^l$ and  $d_l,~l=0,\dots,L$ are nonnegative integers.
Then, substituting the 
above expression in \eqref{nthfocG}, and using  Hirota's $D$-operators\footnote{e.g. $D_tf\cdot g=(\partial_{t'}-\partial_{t})f(t')g(t)|_{t'=t}=f_tg-fg_t.$}
($D_t,~D_x$),  we arrive to the following conditions on ${\bf G}$ and  ${\bf F}$:

\bea\label{Bil_GEa}
&&~{} D_x^2({\bf G}{\bf F}) -2i\mu D_x({\bf G}{\bf F}) = 0,
\eea
\noindent
and
\bea\label{Bil_GE2a}
&&~{} D_t({\bf G}{\bf F}) + \tilde{a}_{n,\mu}D_x^{2n+1}({\bf G}{\bf F}) + \sum_{j=0}^{n-1}\tilde{b}_{n,\mu}D_tD_x^{2(n-j)}({\bf G}{\bf F})=0,\nonu
\eea
where $\tilde{a}_{n,\mu},~\tilde{b}_{n,\mu}$ are coefficients depending on $\mu$ and associated to each member of \eqref{nthfocG}. Then, dividing by ${\bf G}{\bf F}$ the first equation 
\eqref{Bil_GEa}, and taking into account the following identity 
\be\label{identidadfacil}
\frac{D_x^2({\bf G}{\bf F})}{{\bf G}{\bf F}}= \partial_x^2\log({\bf G}{\bf F}) + \left(\partial_x\log(\frac{{\bf G}}{{\bf F}})\right)^2,
 \ee
\noindent
 we obtain:
\bea\label{Bil_GE1a}
& \frac{D_x^2({\bf G}{\bf F})}{{\bf G}{\bf F}} -2i\mu \frac{D_x({\bf G}{\bf F})}{{\bf G}{\bf F}} &= 
\partial_x^2\log({\bf G}{\bf F}) + \left(\partial_x\log(\frac{{\bf G}}{{\bf F}}) \right)^2 -2i\mu\partial_x\log\left(\frac{{\bf G}}{{\bf F}}\right) \nonumber\\
&& = \partial_x^2\log({\bf G}{\bf F}) + \left(\frac{u}{i} \right)^2 - 2\mu u = 0.\nonu
\eea
Hence,
\[
u^2 =  \frac{\partial^2}{\partial x^2}\log({\bf G}\cdot{\bf F}) - 2\mu u.
\]
%
%
\end{proof}

We will also need the following Lemma

\begin{lem}\label{u2hier2}
 Let $B_\mu(t,x)=i\partial_x\log(\frac{F_\mu-iG_\mu}{F_\mu+iG_\mu})=2\partial_x\arctan(\tfrac{G_\mu}{F_\mu})$. Assume that $B_\mu$ satisfies \eqref{B2} and  has  
 velocities $\ga_{2n+1,\mu}$ and $\delta_{2n+1,\mu}$
 
 \be\label{speednthlemma}\begin{aligned}
&\ga_{2n+1,\mu} := -\frac{1}{\bt} \re\Big[\sum_{p=1}^{n}a_{p,n}(\bt + i\al)^{2p+1}\mu^{2(n-p)}\Big],\\
&\delta_{2n+1,\mu} := -\frac{1}{\al} \ima\Big[\sum_{p=1}^{n}a_{p,n}(\bt + i\al)^{2p+1}\mu^{2(n-p)}\Big].\\
\end{aligned}\ee
\noindent

Then $B_\mu$ is a  solution of the  Gardner hierarchy \eqref{nthfocG} for any $t\in\R$.  
\end{lem}
\begin{proof}
By hypothese $B_\mu$ satisfies \eqref{B2}, namely

\[
B_\mu^2 =  \frac{\partial^2}{\partial x^2}\log({\bf G}\cdot{\bf F}) - 2\mu B_\mu.
\]
\noindent
Since, the Gardner hierachy \eqref{nthfocG} is equivalent to system \eqref{Bil_GEa}-\eqref{Bil_GE2a}, it is enough to see that actually $B_\mu$ holds both. First of all, 
since $B_\mu(t,x):= i\partial_x\log(\frac{{\bf G}(t,x)}{{\bf F}(t,x)})$, with ${\bf F} := F_\mu + iG_\mu,~~~ {\bf G} = F_\mu -iG_\mu = {\bf F}^{*}$, and resorting to 
\eqref{identidadfacil}, we rewrite the above identity as:

\[\begin{aligned}
&~B_\mu^2 =  (i\partial_x\log(\frac{{\bf G}(t,x)}{{\bf F}(t,x)}))^2 =  \frac{\partial^2}{\partial x^2}\log({\bf G}\cdot{\bf F}) - \frac{D_x^2({\bf G}{\bf F})}{{\bf G}{\bf F}}\\
&=\frac{\partial^2}{\partial x^2}\log({\bf G}\cdot{\bf F}) - 2\mu(i\partial_x\log(\frac{{\bf G}(t,x)}{{\bf F}(t,x)})),
\end{aligned}\]
\noindent
and therefore
\[\frac{D_x^2({\bf G}{\bf F})}{{\bf G}{\bf F}} - 2\mu(i\partial_x\log(\frac{{\bf G}(t,x)}{{\bf F}(t,x)}))=0,\]
\noindent
and multiplying by ${\bf G}{\bf F}$, we get

\[D_x^2({\bf G}{\bf F}) -2i\mu D_x({\bf G}{\bf F}) = 0,\]
\noindent
which is the first equation \eqref{Bil_GEa} that one obtains after application of the Hirota's operators to the Gardner hierachy \eqref{nthfocG}. In order to describe the time evolution, we should check that \eqref{Bil_GE2a} holds. Instead, we are going to compute directly the velocities, driving the temporal part of \eqref{nthfocG} and which is equivalente to \eqref{Bil_GE2a}. Namely, using a matching method with ansatz $B_\mu$ \eqref{nthbreatherG} and free parameters $\ga_{2n+1,\mu}$ and $\delta_{2n+1,\mu}$, we proceed substituting directly $B_\mu$ \eqref{nthbreatherG} into the 5th and  7th order Gardner equations \eqref{examplesG5}-\eqref{examplesG7}, and
and upon lengthy algebraic manipulations we find that $B_\mu$ is indeed a breather solution of the 5th and  7th order Gardner equations \eqref{examplesG5}-\eqref{examplesG7} respectively, provided that

\be\label{speeds5G}
\begin{aligned}
\delta_{5,\mu} := -\al^4 + 10\al^2\bt^2&-5\bt^4 -10\mu^2(3\bt^2-\al^2) \\
&=  -\frac{1}{\al}\ima\Big[\sum_{p=1}^{2}a_{p,2}(\bt + i\al)^{2p+1}\mu^{2(2-p)}\Big],\\
\ga_{5,\mu} := -\bt^4+10\al^2\bt^2-&5\al^4 -10\mu^2(\bt^2-3\al^2)\\
&= -\frac{1}{\bt} \re\Big[\sum_{p=1}^{2}a_{p,2}(\bt + i\al)^{2p+1}\mu^{2(2-p)}\Big],
\end{aligned}
\ee
\noindent 
with~$a_{1,2}=10,~a_{2,2}=1,$ and

\be\label{speeds7G}
\begin{aligned}
\delta_{7,\mu} := \al^6 &- 21\al^4\bt^2 + 35\al^2\bt^4 - 7\bt^6 + 14\mu^2(-\al^4 + 10\al^2\bt^2-5\bt^4)\\
&- 70\mu^4(3\bt^2-\al^2) =  -\frac{1}{\al}\ima\Big[\sum_{p=1}^{3}a_{p,3}(\bt + i\al)^{2p+1}\mu^{2(3-p)}\Big],\\
&\\
\ga_{7,\mu} :=-\bt^6 &+ 21\al^2\bt^4 - 35\al^4\bt^2+ 7\al^6 + 14\mu^2(-\bt^4+10\al^2\bt^2-5\al^4)\\
&- 70\mu^4(\bt^2-3\al^2) = -\frac{1}{\bt} \re\Big[\sum_{p=1}^{3}a_{p,3}(\bt + i\al)^{2p+1}\mu^{2(3-p)}\Big], \\
\end{aligned}
\ee
\noindent
with~ $a_{1,3}=70,~a_{2,3}=14,~a_{3,3}=1.$ By induction for any  order $2n+1,~n\in\N$, we get that 

\be\label{speedsnthGlemma}
\begin{aligned}
&\ga_{2n+1,\mu} := -\frac{1}{\bt} \re\Big[\sum_{p=1}^{n}a_{p,n}(\bt + i\al)^{2p+1}\mu^{2(n-p)}\Big],\\
&\delta_{2n+1,\mu} := -\frac{1}{\al} \ima\Big[\sum_{p=1}^{n}a_{p,n}(\bt + i\al)^{2p+1}\mu^{2(n-p)}\Big],\\
\end{aligned}\ee
\noindent%
where $a_{p,n}$ is the coefficient of the term $\mu^{2(n-p)}u_{(2p+1)x}$ in the (2n+1)th-order Gardner equation.

\end{proof}

\begin{proof}[Proof of Theorem \ref{MTG0}]

(1) Structure:~ Let $\al, \bt\in \R\backslash\{0\}$ and $\mu \in \R^+\backslash\{0\}$ such that $\Delta=\al^2+\bt^2-4\mu^2>0$,  and $x_1,x_2\in \R$. We check now that
$B_\mu$ in \eqref{nthbreatherG} satisfies, at $t=0$, the identity \eqref{B2} valid for solutions of the Gardner hierarchy \eqref{nthfocG}. For ease of notation let us  use the
following expression for $B_\mu$ \eqref{nthbreatherG}:

\begin{align}\label{notacionBG}
 &B_\mu=2\partial_x\arctan\Big[
 \frac{a_1 \sin \left(y_1\right)-a_2e^{y_2}}{\cosh \left(y_2\right)
 -a_3 (\alpha  \cos \left(y_1\right)-\beta  \sin \left(y_1\right))}\Big]= \frac{H(t,x)}{N(t,x)},
\end{align}

\noindent
with

\be\label{notacionBGhn}\begin{aligned}
 H(t,x)\equiv& ~{}H=2 -a_3 \left(\alpha ^2 a_1+a_2 \left(\beta ^2-\alpha ^2\right) e^{y_2} \sin (y_1)-2 \alpha  a_2
 \beta  e^{y_2} \cos (y_1)\right)\\
 &+\beta  \sinh (y_2) \left(a_2 e^{y_2}-a_1 \sin (y_1)\right)+\cosh (y_2)
 \left(\alpha  a_1 \cos (y_1)-a_2 \beta  e^{y_2}\right),\\
 N(t,x)\equiv&~{}  N=F_\mu^2 + G_\mu^2 = \left(a_2 e^{y_2}-a_1 \sin (y_1)\right)^2+(\cosh (y_2) + a_3 \beta  \sin (y_1)\\
 &-a_3\alpha\cos (y_1))^2,
\end{aligned}\ee
where
\be\label{a1a2a3}
a_1=\frac{\beta  \sqrt{\alpha ^2+\beta ^2}}{\alpha  \sqrt{\Delta }},\quad
a_2=\frac{2 \beta  \mu }{\Delta },\quad
a_3=\frac{2 \beta  \mu }{\alpha  \sqrt{\Delta } \sqrt{\alpha ^2+\beta ^2}}.\ee

\noindent
Then, having in mind notation  $N,N_x,N_{xx}$, $N_{3x},N_{4x}$ and $H,H_x,H_{xx}$,$H_{3x},$ $H_{4x}$ of Appendix \ref{NotacionTheo}, \eqref{B2} simplifies as 

\be\label{B2new} u^2 -  \frac{\partial^2}{\partial x^2}\log(F_\mu^2 + G_\mu^2) + 2\mu u  = \frac{1}{N^2}\Big(H^2 + N_x^2 - N_{xx}N + 2\mu HN\Big).\ee

Finally, substituting explicitly $H's$ and $N's$ terms, we verify, using the symbolic software \emph{Mathematica}, simple trigonometric and hyperbolic identities
 and some rearrangements, that
\be\label{sumaM1M2M3}
H^2 + N_x^2 - N_{xx}N + 2\mu HN=0,
\ee \noindent and we conclude. 

\medskip
%
%
%

\medskip

Now, in order to get a complete dynamical description of $B_\mu$ \eqref{nthbreatherG} at any $t$, we have to obtain the velocities $\ga_{2n+1,\mu}$ and $\delta_{2n+1,\mu}$.

\medskip

(2)~Velocities:~ Now, we determine the velocities $\ga_{2n+1,\mu}$ and $\delta_{2n+1,\mu}$. Using a matching method with ansatz $B_\mu$ as in Lemma \eqref{Bil_GE2a}, we conclude that 

\be\label{speedsnthGth}
\begin{aligned}
&\ga_{2n+1,\mu} := -\frac{1}{\bt} \re\Big[\sum_{p=1}^{n}a_{p,n}(\bt + i\al)^{2p+1}\mu^{2(n-p)}\Big],\\
&\delta_{2n+1,\mu} := -\frac{1}{\al} \ima\Big[\sum_{p=1}^{n}a_{p,n}(\bt + i\al)^{2p+1}\mu^{2(n-p)}\Big],\\
\end{aligned}\ee
\noindent%
where $a_{p,n}$ is the coefficient of the term $\mu^{2(n-p)}u_{(2p+1)x}$ in the (2n+1)th-order Gardner equation.

\end{proof}

\medskip

%
%
%

\begin{rem}
 See Appendix \ref{9mkdv} for a few additional examples of these higher order mKdV and Gardner breather solutions in cases $9th,~11th$ and $13th$ order.
\end{rem}

%
%

\section{Universality in the variational characterization. Proof of Theorem \ref{introMTG0}}\label{Sect:4}

The Gardner hierarchy \eqref{nthfocG}, as being a completely integrable scheme of equations,  has infinitely many
conserved quantities. Some standard conservation laws at the $H^1$-level are the \emph{mass}

\begin{eqnarray}\label{Mass} M[u](t)  :=  \frac 12 \int_\R u^2(t,x)dx =
M[u](0), \end{eqnarray} the \emph{energy}

\be\label{E1N}
E_\mu[u](t)  :=  \int_\R \left(\frac 12 u_x^2 -2\mu u^3  - \frac 12 u^4\right)(t,x)dx = E_\mu[u](0),
\ee

\ms \noindent and the \emph{higher order energy}, defined respectively in $H^2(\R)$,
\be\label{E2}\begin{aligned} F_{\mu}[u](t)  :=   
\int_\R \Big(\frac 12 u_{xx}^2 -10\mu uu_x^2 &+ 10 \mu^2u^4 \\
&- 5u^2u_x^2 + 6\mu u^5 + u^6\Big)(t,x)dx= F_\mu[u](0).
\end{aligned}\ee

%
%

Now,  considering $M[u]$, $E_\mu[u]$ and  $F_\mu[u]$, we define the following Lyapunov functional, as in \eqref{Hmu}:
\begin{equation}\label{LyapunovGE} \mathcal H_\mu[u(t)] := F_\mu[u](t) + 2(\bt^2-\al^2)E_\mu[u](t) + (\al^2 +\bt^2)^2 M[u](t).\end{equation}

\noindent
Therefore, $\mathcal H_\mu[u]$ is  a real-valued conserved quantity, well-defined for $H^2$-solutions of the Gardner hierarchy \eqref{nthfocG}, provided they exist.

\subsection{Proof of Theorem \ref{introMTG0}}
We first prove \eqref{introMTG0a}: we recast the l.h.s. of \eqref{introMTG0a} as follows:
\be \begin{aligned}\label{4ODEBGhierv0}
G_\mu[B_\mu]:=&~{} \partial_{x}^2(B_{\mu,xx} + 6\mu B_\mu^2 + 2B_\mu^3) + 2 \left(\alpha ^2-\beta ^2\right) (B_{\mu,xx} + 6\mu B_\mu^2 + 2B_\mu^3)\\
&~{}+8 \mu  B_\mu B_{\mu,xx}-2 \mu B_{\mu,x}^2 +4B_\mu^2B_{\mu,xx} -2B_\mu B_{\mu,x}^2+\left(\alpha ^2+\beta ^2\right)^2 B_\mu\\
&~{}+40 \mu^2 B_\mu^3+30 \mu B_\mu^4+6 B_\mu^5\\
 =&~{} \partial_{x}^2(B_{\mu,xx} + 6\mu B_\mu^2 + 2B_\mu^3)-2 (\mu +B_\mu )(B_{\mu,x}^2 +4 \mu B_\mu^3+B_\mu^4)\\
&~{} +(\alpha ^2+\beta ^2)^2 B_\mu + (4 B_\mu^2 + 8 \mu B_\mu + 2(\alpha ^2-\beta ^2))(B_{\mu,xx} + 6\mu B_\mu^2 + 2B_\mu^3)\\
&~{} -2 (\mu +B_\mu )(B_{\mu,x}^2 +4 \mu B_\mu^3+B_\mu^4)+(\alpha ^2+\beta ^2)^2 B_\mu.
%
\end{aligned}\ee
\noindent

\medskip

Now we compute explicitly the last line in \eqref{4ODEBGhierv0}. For simplicity, we use the same notation \eqref{notacionBG} and \ref{NotacionTheo} as in the proof of Theorem \ref{MTG0} in \ref{2}. 


\begin{align}\label{notacionBGode}
 &B_\mu(t,x)=\frac{H(t,x)}{N(t,x)},
%
\end{align}

\noindent
with $H,~N$ as in \eqref{notacionBGhn} and $a_i,~i=1,2,3$ already defined in \eqref{a1a2a3}. We first compute the term 
\be\label{h17}\begin{aligned} B_{xx} + 6\mu B^2 + 2B^3 =& \frac{1}{N^3}\Big(2 H^3+6 H^2 \mu  N+H \left(2 N_x^2-N_{xx} N\right)\\
&+N (H_{xx} N-2 H_x N_x)\Big), \end{aligned}\ee
\noindent
and we get
\be\label{h27}
\begin{aligned}
&\partial_x^2(B_{xx} + 6\mu B^2 + 2B^3) = \frac{R_1}{N^5},
\end{aligned}
\ee \noindent
with
\be\label{M1}
\begin{aligned}
R_1:=&\Big(6 H^3 (4 N_x^2-N_{xx} N)-6 H^2 N (6 H_x N_x-N (H_{xx}-2 \mu  N_{xx})-6 \mu  N_x^2)\\
&+H (2 N^2 (6 H_x^2-24 H_x \mu  N_x+4 N_x N_{3x}+3 N_{xx}^2)+N^3 (12 H_{xx} \mu -N_{4x})\\
&+24 N_x^4-36 N_x^2 N_{xx} N)\\
&+N (2 N^2 \left(6 H_x^2 \mu -2 H_x N_{3x}-3 H_{xx} N_{xx}-2 H_{3x} N_x\right)\\
&+12 N_x N (2 H_x N_{xx}+H_{xx} N_x)-24 H_x N_x^3+H_{4x} N^3)\Big).
\end{aligned}
\ee
Moreover, we have that
\be\label{37}
\begin{aligned}
(4 B^2 + 8 \mu B + 2(\alpha ^2-\beta ^2))(B_{xx} + 6\mu B^2 + 2B^3)=\frac{R_2}{N^5},
\end{aligned}
\ee
\noindent
with
\be\label{M2}
\begin{aligned}
&R_2:=2 (2 H^2+4 \mu HN+N^2 (\alpha ^2-\beta ^2))
(2 H^3+6 H^2 \mu  N+H (2 N_x^2-N_{xx} N)\\
&+N (H_{xx} N-2 H_x N_x)),
\end{aligned}
\ee
and
\be\label{47}
\begin{aligned}
-2 (B+\mu )(B_x^2 +4 \mu B^3+B^4)+(\alpha ^2+\beta ^2)^2 B=\frac{R_3}{N^5},
\end{aligned}
\ee
\noindent
with
\be\label{M3}
\begin{aligned}
&R_3:=H N^4 \left(\alpha ^2+\beta ^2\right)^2-2 (H+\mu N) \left(H^4+4\mu H^3N+(H_x N-H N_x)^2\right).
\end{aligned}
\ee

\noindent
Hence, we get the following simplification of \eqref{introMTG0a}:

\begin{align}\label{4ODEBGhier2}
G_\mu[B_\mu] &= \frac{R_1+ R_2 + R_3}{N^5},
\end{align}
\noindent with $R_1,~R_2,~R_3$ in \eqref{M1}, \eqref{M2} and \eqref{M3} respectively. In fact, we verify, using the symbolic software \emph{Mathematica},
that after substituting $H's$ and $N's$ terms explicitly (as they were shown in the Appendix \ref{NotacionTheo}) in \eqref{4ODEBGhier2},  lengthy rearrangements and
basic trigonometric and hyperbolic identities, we get
\be\label{sumaM1M2M3_a}
R_1+ R_2 + R_3=0,
\ee \noindent and we conclude the proof of \eqref{introMTG0a}. 

\medskip

Now, we prove  that $B_\mu$ is actually a critical point of $\mathcal{H_\mu}$. We will show that
\be\label{EE} \mathcal{H_\mu}[B_\mu +z] - \mathcal{H_\mu}[B_\mu]  = \frac 12\mathcal Q[z] + \mathcal N[z], \ee 
with
$\mathcal Q$ being the quadratic form defined in \eqref{Qmu}, and $\mathcal N[z]$ satisfying $|\mathcal N[z] | \leq K\|z\|_{H^2(\R)}^3.$
%
%
We evaluate and expand the Lyapunov functional $H_\mu$ in terms of a perturbation of the breather $B_\mu$, with $z\in H^2(\R)$. A direct computation with 
integration by parts yields
\[\begin{aligned}
&F_{\mu}[B_\mu +z] = ~{} F_{\mu}[B_\mu] \\
&+ \int [B_{\mu, 4x} + 10 \mu B_{\mu, x}^2 + 20\mu B_{\mu}B_{\mu, xx} +10B_{\mu}B_{\mu, x}^2 + 10B_{\mu}^2B_{\mu,x} + 40\mu^2 B_{\mu}^3\\
& + 30\mu B_{\mu}^4 + 6B_{\mu}^5 ] z\\
&+ \frac12 \int \Big(z_{4x} + (20\mu B_{\mu} + 10B_{\mu}^2)z_{xx} - 20(\mu B_{\mu, x} + B_{\mu}B_{\mu,x})z_x \\
& \qquad \qquad + (-10B_{\mu, x}^2 + 120\mu^2 B_{\mu}^2 +120 \mu B_{\mu}^3+ 30B_{\mu}^4)z \Big) z \\
&+\int \left(-10 \mu z z_x^2 - 10 B_{\mu} z z_x - 10 B_{\mu, x}z_x z^2 + 40\mu^2 B_{\mu}z^3 + 60 \mu B_{\mu}^2 z^3 +20 B_{\mu}^3 z^3 \right).
\end{aligned}\]
Similarly,
\[\begin{aligned}
E_\mu[B_{\mu} + z] = E_{\mu}[B_{\mu}]  &- \int ( B_{\mu, xx} + 6\mu B_{\mu}^2 + 2 B_{\mu}^3 ) z - \frac12\int (z_{xx} + (12\mu B_{\mu} \\
&+ 6B_{\mu}^2)z ) z - \int \left(2\mu z^3 + 2 B_{\mu} z^3 + \frac12 z^4 \right),
\end{aligned}\]
and
\[M[B_{\mu} + z] = M[B_{\mu}] + \int B_{\mu} z +  \frac12 \int z^2.
\]
Collecting all, we get that
\[
\mathcal{H}_\mu[B_\mu+z]   =  \mathcal{H_\mu}[B_\mu] + \int_\R G_\mu(B_{\mu}) z + \frac 12\mathcal Q_\mu[z] + \mathcal N_\mu[z],
\]
where the quadratic form 
\be\label{Qmu}
\mathcal Q_\mu [z] := \int z\mathcal L_\mu[z],
\ee
associated to the linearized operator $\mathcal L_\mu$ given by
\begin{equation}\label{linearized operator}
\begin{aligned}
\mathcal L_{\mu} :=&~{} \px^4 + (20 \mu B_{\mu} + 10 B_{\mu}^2 - 2(\beta^2 - \alpha^2)) \px^2 - 20(\mu B_{\mu, x} + B_{\mu}B_{\mu,x}) \px  \\
&+ (-10B_{\mu, x}^2 + 120\mu^2 B_{\mu}^2 +120 \mu B_{\mu}^3+ 30B_{\mu}^4  - 2(\beta^2 - \alpha^2)(12\mu B_{\mu} + 6B_{\mu}^2)\\
&+ (\alpha^2 + \beta^2)^2).
\end{aligned}
\end{equation}
\noindent
Gathering all higher order terms (with respect to $z$) in $N_\mu[z]$ we get 
\[\begin{aligned}
\mathcal N_\mu[z] :=&~{} \int \Big(-10 \mu z z_x^2 - 10 B_{\mu} z z_x - 10 B_{\mu, x}z_x z^2 + 40\mu^2 B_{\mu}z^3 +60 \mu B_{\mu}^2 z^3  \\
&+20 B_{\mu}^3 z^3 \Big) -2(\beta^2 - \alpha^2)\int \left(2\mu z^3 + 2 B_{\mu} z^3 + \frac12 z^4 \right).
\end{aligned}\]
Part (i) above guarantees $\int_\R G_\mu(B_{\mu}) z = 0$, and then we have $\mathcal{H_\mu}'[B_\mu] = 0$. Moreover, from direct estimates, one has 
$\mathcal N_\mu[z] = O(\|z\|_{H^2(\R)}^3),$ and we conclude.
%
%

\begin{proof}\emph{of Corollary \ref{aa4thodemkdv0}}: it follows directly from the above proof for the Gardner case, when we select $\mu=0$.
\end{proof}

\bigskip

\section{Ill-posedness of the Gardner and mKdV hierarchies. Proof of Theorem \ref{ill2n1th0}}\label{ill}

Now, we prove the \emph{ill-posedness} result presented in Theorem \ref{ill2n1th0} for the whole  Gardner hierarchy \eqref{nthfocG}, 
having in mind the explicit breather solution \eqref{nthbreatherG}. Note firstly that from \eqref{nthbreatherG}, the explicit breather solution can be approximated in the limit $\alpha\gg\beta$, namely let $\mu$ 
fixed and suppose that $\frac{\beta}{\alpha}\ll 1$. From \eqref{notacionBG}, we see that the breather solution \eqref{nthbreatherG} reduces to the function

\begin{eqnarray}\label{5thBresG.3} B_{\alpha,\beta,\mu,n}(t,x)\approx
2\beta\cos(\alpha(x+\delta_{2n+1}t))\textrm{sech}(\beta(x+\gamma_{2n+1}t))\end{eqnarray}
or simply
\begin{eqnarray}\label{5thBresG.4} B_{\alpha,\beta,\mu,n}(t,x)\approx
\sqrt{2}\textrm{Re}[e^{i(\alpha(x+\delta_{2n+1}
t))}Q_{\beta}(x+\gamma_{2n+1}t)],\end{eqnarray} where $Q$ denotes
the solution of the nonlinear ODE
\begin{eqnarray}\label{ODE}Q''-Q+Q^3=0,\end{eqnarray} with
\begin{eqnarray}\label{Q}Q(\xi)=\sqrt{2}\textrm{sech}(\xi)\end{eqnarray} and
\begin{eqnarray}\label{Qbeta}Q_{\beta}(\xi)=\beta
Q(\beta\xi).\end{eqnarray}

%
%
%
%

\begin{proof}
We consider the IVP for the (2n+1)th-order Gardner equation with initial data given by the breather solution \eqref{nthbreatherG},

\begin{eqnarray*}\left\{\begin{array}{l}
                                       u_t=-\frac{\partial}{\partial
x}\left(-i\frac{\partial}{\partial
x}+2(\mu+u)\right)\mathcal{L}_n[iu_x+(\mu+u)^2],\\
                                        u(0,x)=B_{\alpha,\beta,\mu,n}(0,x).
                                     \end{array}
\right.\end{eqnarray*}

With $\mu$ fixed, we take the parameter $\alpha$ large enough, such
that $\frac{\beta}{\alpha}\ll 1$. Then, from
(\ref{5thBresG.4}), the initial data reads
\begin{eqnarray}\label{5thBresG.5} B_{\alpha,\beta,\mu,n}(0,x)\approx
\sqrt{2}\textrm{Re}[e^{i\alpha x}Q_{\beta}(x)],\end{eqnarray} with
$Q_{\beta}$ defined in (\ref{Qbeta}). We take
\begin{eqnarray}\label{albt}\beta=\alpha^{-2s}\ \ \textrm{and}\ \ \alpha_1,\alpha_2\sim \alpha.\end{eqnarray}

Observe that $\hat{Q}_{\beta}(\cdot)$ concentrates in the ball $\mathcal B_{\beta}(0)=\{\xi\in\mathbb{R};\ |\xi|<\beta\}$. First, we calculate the $H^s$-norm of two different initial data for the
(2n+1)th-order Gardner equation \eqref{nthfocG} in the regime with $\alpha$ large enough, such that 
$\frac{\beta}{\alpha}\ll 1$:

\begin{eqnarray}\label{estimate}\|B_{\alpha_j,\beta,\mu,n}(0)\|^2_{H^s}\approx\|(1+|\xi|^2)^{s/2}\hat{Q}_{\beta}(\xi-\alpha_i)\|^2_{L^2}\approx C\alpha^{2s}\beta=C,\ \
j=1,2,\end{eqnarray} where $C$ denotes a constant.

Second, we measure the distance between these initial data
\begin{eqnarray}\label{estimate1a}
&&\|B_{\alpha_1,\beta,\mu,n}(0)-B_{\alpha_2,\beta,\mu,n}(0)\|^2_{H^s}\\
&&\qquad \approx
\|(1+|\xi|^2)^{s/2}(\hat{Q}_{\beta}(\xi-\alpha_1)-\hat{Q}_{\beta}(\xi-\alpha_2))\|^2_{L^2}\nonumber\\
&&\qquad  \leq
C\alpha^{2s}\|\hat{Q}_{\beta}(\xi-\alpha_1)-\hat{Q}_{\beta}(\xi-\alpha_2)\|^2_{L^2}\nonu\\
&&\qquad \leq
C\alpha^{2s}\int_{-\infty}^{+\infty}\left|\int_{\xi-\alpha_1}^{\xi-\alpha_2}\frac{d}{d\rho}\hat{Q}_{\beta}(\rho)d\rho\right|^2\
d\xi\nonumber
\end{eqnarray}
Consequently,
\begin{eqnarray}\label{estimate1}
&&\|B_{\alpha_1,\beta,\mu,n}(0)-B_{\alpha_2,\beta,\mu,n}(0)\|^2_{H^s}\\
&&\qquad \leq
C\alpha^{2s}\frac{|\alpha_1-\alpha_2|}{\beta^2}\int_{-\infty}^{+\infty}\int_{\xi-\alpha_1}^{\xi-\alpha_2}|\hat{Q}_{\beta}'(\rho)|^2\
d\rho d\xi\\
&&\qquad \leq
C\alpha^{2s}\frac{|\alpha_1-\alpha_2|}{\beta^2}\left(\int_{\rho+\alpha_2}^{\rho+\alpha_1}d\xi\right)\int_{-\infty}^{+\infty}|\hat{Q}_{\beta}'(\rho)|^2d\rho\nonumber\\
&&\qquad \leq
C\alpha^{2s}\frac{(\alpha_1-\alpha_2)^2}{\beta^2}\beta=C\alpha^{2s}(\alpha_1-\alpha_2)^2\alpha^{2s}=C(\alpha^{2s}(\alpha_1-\alpha_2))^2.\nonumber
\end{eqnarray}

Next, we consider the corresponding solutions
$B_{\alpha_1,\beta,\mu,n}(t)$ and $B_{\alpha_2,\beta,\mu,n}(t)$ at
the time $t=T$. We can see that
\begin{eqnarray}\label{estimate2}\|B_{\alpha_1,\beta,\mu,n}(T)-B_{\alpha_2,\beta,\mu,n}(T)\|^2_{H^s}\approx \alpha^{2s}\|B_{\alpha_1,\beta,\mu,n}(T)-B_{\alpha_2,\beta,\mu,n}(T)\|^2_{L^2}.\end{eqnarray}

From (\ref{5thBresG.4}), if $\alpha$ is large enough,
\begin{eqnarray}\label{apro0} B_{\alpha_j,\beta,\mu,n}(T,x)\approx
\sqrt{2}\textrm{Re}[e^{i(\alpha_j(x+\delta_{2n+1} T))}\beta
Q({\beta}(x+\gamma_{2n+1}T))],\ j=1,2.\end{eqnarray}

Moreover, from \eqref{speednthlemma}, note that
\begin{equation}\label{apro1}
\begin{aligned}
-\gamma_{2n+1}=&(-1)^n(2n+1)\alpha^{2n}\\
& + \sum_{j=0}^{n-1}a_{j,n}P_{(2j)}(\alpha,\beta)\mu^{2(n-j)}+\sum_{j=0}^{n-1}a_{n,n}b_j\beta^{2(n-j)}\alpha^{2j},
\end{aligned}\ee
where $a_{j,n}$ and $b_j$ are convenient real constants and
$P_{(l)}(\alpha,\beta)$ is a polynomial whose its degree is
$l\in\mathbb{N}$. If $\frac{\beta}{\alpha}\ll 1$, we
see that
\begin{eqnarray}\label{apro1.1}\gamma_{2n+1}\sim
(-1)^{n+1}(2n+1)\alpha^{2n}\end{eqnarray} and
\begin{eqnarray}\label{apro2}\alpha_1^{2n}-\alpha_2^{2n}=\left(\sum_{j=0}^{2n-1}\alpha_1^{2n-1-j}\cdot\alpha_2^j\right)(\alpha_1-\alpha_2)\sim(\alpha_1-\alpha_2)\alpha^{2n-1}.\end{eqnarray}
The information above shows that $B_{\alpha_j,\beta,\mu,n}(T)$,
$j=1,2$, concentrates in the ball
$$\mathcal B_{\beta^{-1}}((-1)^n(2n+1)\alpha_j^{2n}T),\ \ j=1,2.$$ So, we
basically have disjoint supports if
\begin{eqnarray}\label{apro4}\alpha^{2n-1}(\alpha_1-\alpha_2)T\gg\beta^{-1}=\alpha^{2s}.\end{eqnarray}

Under this condition, we have that
\be\label{apro5}\begin{aligned}&\|B_{\alpha_1,\beta,\mu,n}(T)-B_{\alpha_2,\beta,\mu,n}(T)\|^2_{L^2}\\
&\approx\|B_{\alpha_1,\beta,\mu,n}(T)\|^2_{L^2}+\|B_{\alpha_2,\beta,\mu,n}(T)\|^2_{L^2}\approx\beta\end{aligned}\ee
and
\be\label{apro6}\begin{aligned}\|B_{\alpha_1,\beta,\mu,n}(T)-B_{\alpha_2,\beta,\mu,n}(T)\|^2_{H^s}\geq C\alpha^{2s}\beta=C.\end{aligned}\ee

If we select
\begin{eqnarray}\label{select}\alpha_1=\alpha+\frac{\delta}{2\alpha^{2s}},\ \ \alpha_2=\alpha-\frac{\delta}{2\alpha^{2s}},\ \
\alpha_1-\alpha_2=\frac{\delta}{\alpha^{2s}},\end{eqnarray}
we have that
\begin{eqnarray}\label{select2}(\alpha^{2s}(\alpha_1-\alpha_2))^2=\delta^2\end{eqnarray}
and, from (\ref{apro4}),
\begin{eqnarray}\label{select3}\alpha^{2n-1}\frac{\delta}{\alpha^{2s}}T\gg\alpha^{2s}.\end{eqnarray}
Finally, from (\ref{select3}),
\begin{eqnarray}\label{select4}T\gg\frac{\alpha^{4s-2n+1}}{\delta}.\end{eqnarray}

Since $s<\frac{2n-1}{4}$, given $\delta,T>0$, we can
choose $\alpha$ so large that (\ref{select4}) is still valid, and
then (\ref{apro6}) does not satisfy uniform continuity. The proof is
complete.
\end{proof}

\begin{cor}[Ill-Posedness of the $(2n+1)$th-order mKdV equation]\label{ill2n1thmkdv}
If $s<\frac{2n-1}{4}$, the mapping data-solution $u_0\rightarrow u(t)$, with $u(t)$ a solution of the IVP for the
(2n+1)th-order mKdV equation \eqref{nthfocmkdv} is not uniformly continuous.
\end{cor}

\begin{proof}
 Selecting $\mu=0$ in the above theorem, we get the result.
\end{proof}

\bigskip

\section{Remarks on periodic breathers for the 5th and 7th mKdV equations}\label{periodicB}

\medskip

In this section we provide further details on the introduction of periodic in space breathers for the mKdV hierarchy, namely, Theorem \ref{introperiodicmkdv}. 

\medskip

\subsection{5th and 7th order mKdV} We consider now, from \eqref{examplesG5}  and \eqref{examplesG7} when $\mu=0$, the  periodic case of the 5th-order mKdV:
\be\begin{aligned}\label{5mkdv}
& u_{t} +(u_{4x} + f_5(u))_x=0,\qquad f_5(u):=~10uu_x^2 +10u^2u_{xx} + 6u^5,
\end{aligned}\ee
\noindent
and the 7th-order mKdV 
\be\begin{aligned}\label{7mkdv}
u_{t} +(u_{6x} &+ f_7(u))_x=0,\\
&\qquad f_7(u):=~14u^2u_{4x} +56uu_xu_{3x} + 42uu_{xx}^2 + 70u_x^2u_{xx} \\
&\quad\quad + 70u^4u_{xx} + 140u^3u_x^2 + 20u^7.
\end{aligned}\ee
\subsection{Standard and new mKdV periodic breathers}
A family of periodic breathers (named KKSH breathers) for the \emph{classical} mKdV equation was found by Kevrekidis et al  by using elliptic 
functions and a matching of free parameters (see \cite{KKS1,KKS2,AMP2} for further reading). For the higher order mKdV equations \eqref{5mkdv} 
and \eqref{7mkdv}, an equivalent expression of periodic  breathers is  available, by following a similar 
matching of parameters. Namely, we consider here the 5th and 7th mKdV equations \eqref{5mkdv} \eqref{7mkdv} where now
\[
u:  \R_t\times \T_x \mapsto \R_x,
\]
is periodic in space, and $\T_x =\T = \R/ L\Z  =(0,L)$ denotes a torus with period $L$, to be fixed later. Higher order periodic cases can also be described 
but for the shake of simplicity, we will keep our discussion with these 5th and 7th orders.

\medskip

We refer the reader to \cite{AS,By} for a more detailed account on the Jacobi elliptic functions $\sn$ and $\nd$ presented below. The proof of Theorem \ref{introperiodicmkdv} is essentially contained in the following 

\begin{prop}[Periodic breathers of 5th and 7th mKdV equations]\label{DEF_B_per}
Given $\al,\bt>0$, $x_1,x_2\in \R$ and $k,m\in [0,1],$ the following is satisfied.

\ben
\item Periodic breather solutions of the 5th and 7th mKdV equations \eqref{5mkdv} \eqref{7mkdv}, are given by the explicit formula (see
\cite{KKS1} and \cite{AMP2} for checking and comparison reasons)
\be\label{57Bper}
\begin{aligned}
B= &~{} B(t,x; \al, \bt, k, m, x_1,x_2)   \\
:= &~{} \partial_x \tilde B :=
2\partial_x \Big[ \arctan \Big( \frac{\bt}{\al}\frac{\sn(\al y_1,k)}{\nd(\bt y_2,m)}\Big) \Big],
\end{aligned}
\ee
with $\sn(\cdot ,k)$ and $\nd (\cdot , m)$ the standard Jacobi elliptic functions of elliptic modulus $k$ and $m$, respectively, but now
\be\label{Y1_Y2}
y_1:=x+\delta_{i,m} t + x_1,  \quad y_2:=x+\ga_{i,m} t +x_2,\quad i=5,7.
\ee
\item The velocities $(\delta_{5,m},\ga_{5,m})$ in the 5th order case and  $(\delta_{7,m},\ga_{7,m})$ in the 7th order case are given by, respectively:
\be\label{perspeeds5}
\begin{aligned}
\delta_{5,m} := &~{}  -\al^4(k^2-26k+1) \\
&~{} +10\al^2\bt^2(1+k)(2-m)-5\bt^4(m^2-16m+16),\\
\ga_{5,m} :=&~{} -\bt^4(m^2+24m-24) \\
&~{}+10\al^2\bt^2(1+k)(2-m)-5\al^4(k^2+14k+1),
\end{aligned}
\ee
and
\be\label{perspeeds7}
\begin{aligned}
\delta_{7,m} := &~{}\al^6(k^3+135k^2+135k+1) \\
&~{}+ 21\al^4\bt^2(-2+k^2(m-2)+m+2k(7m-6))\\
&~{}+ 7\al^2\bt^4(1+k)(5m^2-24m+24) \\
&~{}+ 7\bt^6(m^3-2m^2+48m-48),\\
\ga_{7,m} :=&~{} -\bt^6(-m^3-254m^2-2256m+2512)\\
&~{} + 7\al^2\bt^4(1+k)(3m^2+88m-88) \\
&~{}+ 7\al^4\bt^2(5(k^2+1)(m-2)+k(70m+292))\\
&~{}+ 7\al^6(k^3+135k^2+135k+1).
\end{aligned}
\ee
\item Additionally, in order to be a periodic solution of 5th and 7th-mKdV equations
(and also for the classical mKdV), the parameters $m,k,\al$ and $\bt$ must satisfy the following commensurability conditions on the spatial periods
\be\label{Cond1}
\frac{\bt^4}{\al^4}=\frac{k}{1-m}, \qquad K(k)=\frac{\al}{2\bt} K(m),
\ee
where $K$ is the complete elliptic integral of the first kind, defined as \cite{By}
\be\label{Kint}
\begin{aligned}
K(r):=&~{}  \int_0^{\pi/2}(1-r\sin^2(s))^{-1/2}ds \\
=&~{}  \int_0^1((1-t^2)(1-rt^2))^{-1/2}dt,
\end{aligned}
\ee
and which satisfies
\[
K(0) =\frac \pi2 \quad \hbox{ and }\quad \lim_{k\to 1^-}K(k)=\infty.
\]
\item The spatial period of the breather is given by
\be\label{Largo}
L:=\frac{4}{\alpha}K(k) =\frac{2}{\bt} K(m).
\ee
\een
\end{prop}

\begin{rem}
Note that conditions \eqref{Cond1}  formally imply that the periodic breather $B$ \eqref{57Bper} 
has only four independent parameters (e.g. $\bt,~k$ and translations $x_1,~x_2$). Additionally, if we assume that
the ratio $\bt/\al$ stays bounded, we have that $k$ approaches $0$ as $m$ is close to $1$. Using this information, the standard non
periodic  5th and 7th-mKdV breathers can be formally recovered as the limit of very large spatial period $L\to +\infty$, obtained e.g.
if $k\to 0.$
\end{rem}

\begin{rem}
Note that these breathers can be written using only two parametric variables,
say $\bt$ and $k$, and have a characteristic period $L=L(\bt,k)$, with $L\to +\infty$ as $k\to 0$.
Moreover, compare the periodic higher order velocities $(\delta_{i,m},\ga_{i,m}),~i=5,7$ above, with the equivalent periodic ones in the simpler classical
mKdV case
\cite[Def.1.1]{AMP2}:
\be\label{perDG}
 \delta := \al^2(1+k) +3\bt^2(m-2), \quad \hbox{ and }\quad \ga := 3\al^2(1+k) +\bt^2(m-2),
 \ee
 \noindent
and with  velocities  $(\delta_i,\ga_i),~i=5,7,$  (\eqref{speeds5G}-\eqref{speeds7G} when $\mu=0$) in the non periodic case:
\be\label{speeds579}
\begin{aligned}
&\delta_5 := -\al^4+10\al^2\bt^2-5\bt^4,\quad\ga_5 :=-\bt^4+10\al^2\bt^2-5\al^4,\\
\end{aligned}\ee
\noindent  and
\be\label{speeds579-1}
\begin{aligned}
\delta_7 := &~{} \al^6-21\al^4\bt^2+35\al^2\bt^4-7\bt^6,\\
\ga_7 :=&~{} -\bt^6+21\al^2\bt^4-35\al^4\bt^2+7\al^6.
\end{aligned}\ee
\end{rem}

\appendix
\section{Notation in proof of Theorem \ref{introMTG0}}\label{NotacionTheo}

We will use the following notation for the sake of simplicity:

\begin{align}\label{notacionN17}
&  N_x:= 2\,\alpha\, \left({a_{{1}}}^{2}-{\alpha}^{2}{a_{{3}}}^{2}+{a_{{3}}}^{2}{\beta}^{ 2}
\right)\sin \left( y_{{1}} \right) \cos \left( y_{{1}} \right)-2\,\alpha\,a_{{1}}a_{{2}}{{\rm e}^{y_{{2}}}}\cos \left(
y_ {{1}} \right)\nonu\\
&-2\,a_{{1}}a_{{2}}\beta\,{{\rm e}^{y_{{2}}}}\sin\left( y_{{1}} \right) +2\,{a_{{2}}}^{2}\beta\,{{\rm e}^{2\,y_{{2}}}}
+2\,{\alpha}^{2}{a_{{3}}}^{2}\beta\, \left( \sin \left( y_{{1}} \right) \right) ^{2}\nonu\\
&-2\,{\alpha}^{2}{a_{{3}}}^{2}\beta\, \left( \cos \left( y_{{1}} \right)  \right) ^{2}+2\,{\alpha}^{2}a_{{3}}\sin \left(y_{{1}} \right) 
\cosh \left( y_{{2}}\right)-2\,\alpha\,a_{{3}}\beta\,\cos \left( y_{{1}} \right) \sinh \left( y_{{2}} \right) \nonu\\
&  +2\,\alpha\,a_{{3}}\beta\,\cos \left( y_{{1}} \right) \cosh \left( y_{{2}} \right)
+2\, a_{{3}}{\beta}^{2}\sin \left( y_{{1}} \right) \sinh \left(y_{{2}} \right) +2\,\beta\,\sinh \left( y_{{2}} \right) \cosh \left( y_{{2}}
 \right),
\end{align}
 \begin{align}\label{notacionN18}
&  N_{xx}:=8\,{\alpha}^{3}{a_{{3}}}^{2}\beta\,\sin \left(
y_{{1}} \right)\cos \left( y_{{1}} \right) -4\,\alpha\,a_{{1}}a_{{2}}\beta\,{{\rm e}^{y_{{2}}}}\cos \left( y_{{1}} \right)
+4\,{a_{{2}}}^{2}{\beta}^{2}{{\rm e}^{2\,y_{{2}}}}\nonumber\\
&+2\,{\alpha}^{2} \left( {a_{{1}}}^{2}-{\alpha}^{2}{a_{{3}}}^{2}+{a_{{3}}}^{2}{\beta}^{2} \right)
 \left( \cos \left( y_{{1}} \right)  \right) ^{2} -2\,{\alpha}^{2} \left( {a_{{1}}}^{2}-{\alpha}^{2}{a_{{3}}}^{2}+{a_{{3}}}^{ 2}{\beta}^{2}
\right)  \left( \sin \left( y_{{1}} \right)  \right) ^{2 }\nonu\\
&-2\,a_{{3}}\beta\,\left( {\alpha}^{2}-{\beta}^{2} \right)\sin\left( y_{{1}} \right) \cosh \left( y_{{2}}\right) +4\,\alpha\,a_{{3}}{\beta}^{2}\cos
 \left( y_{{1}} \right) \sinh \left( y_{{2}} \right) \nonu\\
&+4\,{\alpha}^{2}a _{{3}}\beta\,\sin \left( y_{{1}} \right) \sinh\left( y_{{2}} \right) +2\,{\beta}^{2} \left( \cosh \left( y_{{2}}\right) \right) ^{2}+2\,{ \beta}^{2} \left( \sinh \left( y_{{2}}
\right) \right)^{2},\nonu\\
&+2\,a_{{1} }a_{{2}}\left( {\alpha}^{2}-{\beta}^{2} \right){{\rm e}^{y_{{2}}}}\sin \left( y_{{1}} \right) 
 +2\alpha\,a_{{3}}\,\left( {\alpha}^{2}-{\beta}^{2} \right)\cos \left( y_{{1}} \right)\cosh
\left( y_{{2}} \right)
\end{align}
\begin{align}\label{notacionN19}
&  N_{3x}:=-8\,{\alpha}^{3}\left( {a_{{1}}}^{2}-{\alpha}^{2}{a_{{3}}}^{2}+{a_{{3}}}^{2}
{\beta}^{2} \right)\sin \left( y_{{1}} \right)\cos \left( y_{{1}} \right)+8\,{a_{{2}}}^{2}{\beta}^{3}{{\rm e}^{2\,y_{{2}}}}\nonu\\
&+2\,a_{{1}}a_{{2}}\alpha\,\left({\alpha}^{2}-3\,{\beta}^{2}\right){{\rm e}^{y_{{2}}}}\cos \left( y_{{1}} \right)
+8\,{\alpha}^{4}{a_{{3}}}^{2}\beta\, \left( \cos \left( y_{{1}} \right)  \right) ^{2}-8\,{\alpha}^{4}{a_{{3}}}^{2}\beta\left( \sin \left( y_{{1}} \right)
\right)^{2}\nonu\\
 &+8\,{\beta}^{3}\sinh \left( y_{{2}} \right) \cosh \left( y_{{2}} \right)-2\,{\alpha} ^{2}a_{{3}}\left({\alpha}^{2}-3\,{\beta}^{2}\right)\sin
\left( y_{{1}} \right)\cosh \left( y_{{2}}\right)\nonu\\
&-2\,{\beta}^{2}a_{{3}}\left( 3\, {\alpha}^{2}-{\beta}^{2}
\right)\sin \left( y_{{1}} \right)\sinh \left( y_{{2}}\right)+2\,a_{{1}}a_{{2}}
\beta\,\left( 3\,{\alpha}^{2}-{\beta}^{2} \right){{\rm e}^{y_{{2}}}}\sin \left( y_{{1}} \right)\nonu\\
&-2\,\alpha\,a_{{3}}\beta\,\left( {\alpha}^{2}-3\,{ \beta}^{2}\right)\cos
 \left( y_{{1}}\right)\cosh \left( y_{{2}}\right)+2\,a_{{3}}\alpha\,\beta\,\left(3\,{\alpha}^{2}-{\beta}^{2} \right)
 \cos \left( y_{{1}} \right)\sinh \left( y_{{2}} \right)
\end{align}
\begin{align}\label{notacionN20}
&  N_{4x}:=-32\,{\alpha}^{5}{a_{{3}}}^{2}\beta\,\cos \left(y_{{1}} \right) \sin \left( y_{{1}} \right) 
+8\,\alpha\,a_{{1}}a_{{2}}\beta\,\left( {\alpha}^{2}-{\beta}^{2} \right){{\rm e}^{y_{{2}}}}\cos \left( y_{{1}} \right)  \nonu\\
 &-2\,a_{{1}}a_{{2}}\left({\alpha}^{4}+{\beta}^{4}-6\,{\alpha}^{2}{\beta}^{2} \right){{\rm e}^{y _{{2}}}}\sin \left( y_{{1}}\right)
 +16\,{a_{{2}}}^{2}{\beta}^{4}{{\rm e}^{2\,y_{{2}}}}+8\,{\beta}^{4} \left( \sinh \left( y_{{2}} \right)  \right) ^{2}\nonu\\
&+8\,{\alpha}^{4}\left({a_{{1}}}^{2}-{\alpha}^{2}{a_{{3}}}^{2}+{a_{{3}}}^{2}{ \beta}^{2}\right) \left( \sin \left( y_{{1}} \right)  \right)
^{2}+8\,{\beta}^{4} \left( \cosh \left( y_{{2}} \right)  \right) ^{2}\nonu\\
 &-2\,\alpha\,a_{{3}}\left( {\alpha}^{4}+{\beta}^{4}-6\,{\alpha}^{2}{\beta}^{2} \right) \cos\left( y_{{1}} \right)\cosh \left( y_{{2}} \right)\nonu\\
 &-8\,\alpha\,a_{{3}}{\beta}^{2}\left( {\alpha}^{2}-{\beta}^{2} \right) \cos \left( y_{{1}} \right) \sinh\left( y_{{2}} \right) 
 -8\,{\alpha}^{4}\left( {a_{{1}}}^{2}-{\alpha}^{2}{a_{{3}}}^{2}+{a_{{3}}}^{2}{\beta}^{2}
 \right) \left( \cos \left( y_{{1}} \right) \right)^{2}\nonu\\
&+2\,a_{{3}}\beta\,\left({\alpha}^{4}+{\beta}^{4}-6\,{\alpha}^{2}{\beta}^ {2} \right)\sin\left( y_{{1}} \right) \cosh \left( y_{{2}} \right)-8\,{\alpha}^{2}a_{{3}}\left( {\alpha}^{2}-
{\beta}^{2} \right)\sin \left( y_{ {1}} \right) \beta\,\sinh \left(y_{{2}} \right)
\end{align}
\noindent
and
\begin{align}\label{notacionHN7}
&  H_x:=2\, \left({\alpha}^{2}+{\beta}^{2}\right)  \left( a_{
{2}}a_{{3}}\alpha\,{{\rm e}^{y_{{2}}}}\cos \left( y_{{1}} \right)
-a_{{1}}\cosh \left( y_{{2 }} \right) \sin \left( y_{{1}} \right)
-a_{{2}}a_{{3}}\beta\,{ {\rm e}^{y_{{2}}}}\sin \left(y_{{1}} \right)  \right),
\end{align}
\begin{align}\label{notacionHN7a}
&  H_{xx}:=-2\,\left({\alpha}^{2}+{\beta}^{2}\right)(\alpha\,a_{{1}} \cosh \left(y_{{2}} \right)\cos \left( y_{{1}}\right)+a_{{1}}\beta\,\sinh
 \left( y_{{2}} \right) \sin \left( y_{{1}} \right)\quad\quad\quad\quad\nonu\\
&+a_{{2}}a_{{3}}\left({\alpha}^{2}+{\beta}^{2}\right){{\rm e}^{y_{{2}}}}\sin \left(y_{{1}} \right)),
\end{align}
\begin{align}\label{notacionHN7b}
&  H_{3x}:=2\,a_{{1}}\left( {\alpha}^{4}-{\beta}^{4}
\right)\cosh \left( y_{{2}} \right)\sin \left( y_{{1}}\right)
-4\,a_{{1}}\alpha\,\beta\,\left( {\alpha}^{2}+{\beta}^ {2}
\right)\sinh\left(y_{{2}}\right)\cos\left( y_{{1}}\right)\quad\quad\nonu\\
&-2\,a_{{2}}a_{{3}}\beta\,\left( {\alpha}^{2}+{\beta}^ {2} \right) ^
{2}{{\rm e}^{y_{{2}}}}\sin \left(y_{{1}}
\right)-2\,a_{{2}}a_{{3}}\alpha\,  \left( {\alpha}^{2}+{\beta}^ {2}
\right) ^{2}{{\rm e}^{y_{{2}}}}\cos\left(y_{{1}}\right),
\end{align}
\begin{align}\label{notacionHN7c}
&  H_{4x}:=2\,a_{{1}}\alpha\,\left({\alpha}^{4}-2\,{\alpha}^{2}{\beta}^{2}-3\,{\beta}^{4 } \right)\cosh\left( y_{{2}} \right)\cos \left( y_{{1}}
 \right) \nonumber\\
&-4\,a_{{2}}a_{{3}}\alpha\,\beta\,\left( {\alpha}^{2}+{\beta}^{2}\right)^2{{\rm e}^{y_{{2}} }}\cos \left( y_{{1}} \right)\nonu\\
&+2\,a_{{1}}\beta\, \left(3\,{\alpha}^{4}+2\,{ \alpha}^{2}{\beta}^{2}-{\beta}^{4}\right)\sinh\left( y_{{2}}\right)\sin \left( y _{{1}}
\right)\nonu\\
&+2\,a_{{2}}a_{{3}}\left({\alpha}^{6}+{\alpha}^{4}{\beta}^{2}-{\alpha}^{ 2}{\beta}^{4}-{\beta}^{6}\right){{\rm e}^{y_{{2}}}} \sin \left( y_{{1}} \right).
\end{align}
\noindent

\section{Additional higher order mKdV and Gardner equations}
For the sake of completeness and forthcoming work by elsewhere, we list below higher order members of the mKdV and Gardner hierarchies \eqref{nthfocmkdv}-\eqref{nthfocG}.
We start with the  9th order mKdV equation which is written as follows
\be\label{9mkdv}
\begin{aligned}
u_t &+ \partial_x\Big(u_{8x} + 18u^2u_{6x} + 108uu_xu_{5x} +228uu_{2x}u_{4x} + 210u_x^2u_{4x} + 126u^4u_{4x} \\
&+ 138u(u_{3x})^2 + 756u_xu_{2x}u_{3x} + 1008u^3u_xu_{3x} + 182u_{2x}^3 + 756u^3u_{2x}^2 \\
&+ 3108u^2u_x^2u_{2x} + 420u^6u_{2x} + 798uu_x^4 + 1260u^5(u_x)^2 + 70u^9\Big)=0.
\end{aligned}\ee

The 11th order mKdV equation is written as follows

\be\label{11mkdv}
\begin{aligned}
&u_t + \partial_x\Big(
u_{10x}+22 u^2 u_{8x}+198 u^4u_{6x}+924 u^6 u_{4x}+506 u u_{4x}^2 +3036 u^3 u_{3x}^2\\
&+2310 u^8 u_{xx}+8316 u^5 u_{xx}^2+9372 u^2 u_{xx}^3+9240 u^7 u_x^2 +26796 u^3 u_x^4\\
&+176 u u_x u_{7x}+484 uu_{xx} u_{6x}+462 u_x^2 u_{6x}+836 u u_{3x}u_{5x}+2376 u^3 u_xu_{5x}\\
&+5016 u^3 u_{xx} u_{4x}+2706 u_{xx}^2 u_{4x}+11220 u^2 u_x^2 u_{4x}+3498 u_{xx}u_{3x}^2+11088 u^5 u_x u_{3x}\\
&+54516 u^4 u_x^2 u_{xx}+44748 u u_x^2 u_{xx}^2+13398 u_x^4 u_{xx}+2376 u_x u_{xx} u_{5x}\\
&+21120 uu_x^3 u_{3x}+3696 u_x u_{3x} u_{4x}+39336 u^2 u_x u_{xx} u_{3x}+252 u^{11}\Big)=0.
\end{aligned}\ee

We finally present the 13th order mKdV equation 

\be\label{13mkdv}
\begin{aligned}
&u_t+\partial_x\Big(u_{12x} + 1846\, uu_{5x}^{2}+191620 u^{4}u_{xx}^{3}+924u^{13}+30888u^{5}u_{x}u_{5x}\\
&+1733160u^{3}u_{xx}^{2}u_{x}^{2}+823680u^{3}u_{x}^{3}u_{3x}+1398540u^{2}u_{x}^{4}u_{xx}\\
&+648648u^{6}u_{x}^{2}u_{xx}+96096u^{7}u_{x}u_{3x}+3172u u_{4x}u_{6x}+1976u u_{3x}u_{7x}\\
&+884u u_{xx}u_{8x}+260u  u_{x}u_{9x}+324324u u_{x}^{2}u_{3x}^{2}+78936uu_{x}^{3}u_{5x}\\
&+511368uu_{x}^{2}u_{xx}u_{4x}+919776uu_{x}u_{xx}^{2}u_{3x}+566280u_{x}^{3}u_{3x}u_{xx}\\
&+17160u_{x}u_{4x}u_{5x}+5720u_{x}u_{xx}u_{7x}+12012u_{x}u_{3x}u_{6x}+33176\,u_{xx}u_{3x}u_{5x}\\
&+21736u^{3}u_{3x}u_{5x}+12584u^{3}u_{xx}u_{6x}+4576u^{3}u_{x}u_{7x}+29172u^{2}u_{x}^{2}u_{6x}\\
&+231660u^{4}u_x^{2}u_{4x}+65208u^{5}u_{xx}u_{4x}+78078u_{x}^{4}u_{4x}+403260u_{x}^2u_{xx}^{3}\\
&+8866u_{xx}^{2}u_{6x}+20306u_{xx}u_{4x}^{2}+858u_{x}^{2}u_{8x}+26598u_{3x}^{2}u_{4x}\\
&+468468u^{5}u_{x}^{4}+60060u^{9}u_{x}^{2}+157300uu_{x}^{6}+13156u^{3}u_{4x}^{2}+286u^{4}u_{8x}
\end{aligned}\ee
\be\begin{aligned}
&+39468u^{5}u_{3x}^{2} +1716u^{6}u_{6x}+12012u^{10}u_{xx}+72072u^{7}u_{xx}^{2}+6006u^{8}u_{4x}\\
&+108966uu_{xx}^{4}+26u^{2}u_{10x}+156156u^2u_{xx}^{2}u_{4x}+197340u^{2} u_{xx}u_{3x}^{2}\\
&+219648u^{2}u_{x}u_{3x}u_{4x}+144144u^{2}u_{x}u_{xx}u_{5x}+806520u^{4}u_{x}u_{xx}u_{3x}\Big)=0.
\end{aligned}\ee

Now, associated with these additional higher order mKdV equations, we provide the corresponding breather solutions, as in \eqref{intronthBresmkdv}:

\begin{defn}[9th-11th-mKdV breathers]\label{911breather} Let $\al, \bt >0$
and $x_1,x_2\in \R$. The real-valued  breather solution associated to the 9th-11th-mKdV equations 
\eqref{9mkdv}-\eqref{11mkdv}-\eqref{13mkdv} are given explicitly by the formula
\be\label{911Bre}B\equiv B_{\al, \bt}(t,x;x_1,x_2)   
 :=   2\partial_x\Bigg[\arctan\Big(\frac{\bt}{\al}\frac{\sin(\al y_1)}{\cosh(\bt y_2)}\Big)\Bigg],
\ee
with $y_1$ and $y_2$
\be\label{y1y2GE9}
\begin{aligned}
&y_1 = x+ \delta_i t + x_1, \quad y_2 = x+ \ga_i t + x_2,~~~\end{aligned}\ee
\noindent 
and with  velocities  $(\delta_i,\ga_i),~i=9,11$
\be\label{speeds9}
\begin{aligned}
&\delta_9 := -\al^8 + 36\al^6\bt^2 - 126\al^4\bt^4 + 84\al^3\bt^6-9\bt^8,\qquad\qquad\quad\quad\\
&\ga_9 :=-\bt^8+36\al^2\bt^6-126\al^4\bt^4 + 84\al^6\bt^2-9\al^8,\qquad\qquad\quad\quad\\
&\\
\end{aligned}\ee
\noindent and
\be\label{speeds11}
\begin{aligned}
& \delta_{11} = \alpha ^{10}-55 \alpha ^8 \beta ^2+330 \alpha ^6 \beta ^4-462 \alpha ^4 \beta ^6 +165 \alpha ^2 \beta ^8-11\bt^{10},\\
 & \gamma_{11} = 11 \alpha ^{10} -165 \alpha ^8 \beta ^2+462 \alpha ^6 \beta ^4-330 \alpha ^4 \beta ^6+55 \alpha ^2 \beta ^8-\beta ^{10}.
\end{aligned}\ee
\end{defn}

\medskip

Finally, and  directly from definition of the Gardner hierarchy \eqref{nthfocG} with $n=4$ or alternatively, from the 9th and 11th order mKdV equations
\eqref{9mkdv}-\eqref{11mkdv} with the transformation $u\rightarrow \mu + u$, we present the 9th order Gardner equation:
\be\label{9gardner}
\begin{aligned}
u_t &+ \partial_x\Big(u_{8x} + 18(\mu + u)^2u_{6x} + 108(\mu + u)u_xu_{5x} +228(\mu + u)u_{2x}u_{4x} + 210u_x^2u_{4x} \\
& + 126(\mu + u)^4u_{4x} + 138(\mu + u)u_{3x}^2 + 756u_xu_{2x}u_{3x} + 1008u^3u_xu_{3x} + 182u_{2x}^3 \\
&+ 756(\mu + u)^3u_{2x}^2  + 3108(\mu + u)^2(u_x)^2u_{2x} + 420(\mu + u)^6u_{2x} + 798(\mu + u)u_x^4\\
&+ 1260(\mu + u)^5u_x^2 + 70(\mu + u)^9\Big)=0,
\end{aligned}\ee
\noindent
and the 11th order Gardner equation

\be\label{11gardner}
\begin{aligned}
&u_t + \partial_x\Big(
u_{10x}+22 (\mu + u)^2 u_{8x}+198 (\mu + u)^4u_{6x}+924 (\mu + u)^6 u_{4x}+506 (\mu + u)u_{4x}^2\\
& +3036 (\mu + u)^3 u_{3x}^2 +2310 (\mu + u)^8 u_{xx} + 8316 (\mu + u)^5 u_{xx}^2 + 9372 (\mu + u)^2 u_{xx}^3\\
&+9240 (\mu + u)^7 \left(u_x\right)^2 +26796 (\mu + u)^3 u_x^4+176 (\mu + u)u_x u_{7x}+484 (\mu + u)u_{xx} u_{6x}\\
&+462 u_x^2 u_{6x}+836 (\mu + u)u_{3x}u_{5x} + 2376 (\mu + u)^3 u_xu_{5x}+5016 (\mu + u)^3 u_{xx} u_{4x}\\
&+ 2706 u_{xx}^2 u_{4x}+11220 (\mu + u)^2 u_x^2 u_{4x} + 3498 u_{xx}u_{3x}^2 + 11088 (\mu + u)^5 u_x u_{3x}\\
&+54516 (\mu + u)^4 u_x^2 u_{xx} + 44748 (\mu + u) u_x^2 u_{xx}^2 + 13398 u_x^4 u_{xx} + 2376 u_x u_{xx} u_{5x}\\
&+21120 (\mu + u)u_x^3 u_{3x} + 3696 u_x u_{3x} u_{4x}+39336 (\mu + u)^2 u_x u_{xx} u_{3x}\\
&+252 (\mu + u)^{11}\Big)=0.
\end{aligned}\ee

\medskip
The corresponding breather solutions of the 9th and 11th order Gardner equations \eqref{9gardner}-\eqref{11gardner} are the following

\begin{defn}[9th-11th-Gardner breathers]\label{9GardnerB} Let $\al, \bt, \mu$ as in definition \ref{nthbreatherG}
and $x_1,x_2\in \R$. Then the real-valued  breather solution associated to the 9th and 11th-Gardner equations
\eqref{9gardner}-\eqref{11gardner} is given explicitly by the same formula as in \eqref{nthbreatherG} but respectively 
with  velocities  $(\delta_9,\ga_9)$ and $(\delta_{11},\ga_{11})$:

\be\label{speeds9ga}
\begin{aligned}
&\delta_9 := \alpha (\alpha ^8-18 \alpha ^6 \left(2 \beta ^2+\mu ^2\right)+126 \alpha ^4 \left(\beta ^4+3 \beta ^2 \mu ^2+\mu ^4\right)\\
&-42 \alpha ^2 \left(2 \beta ^6+15 \beta ^4 \mu ^2 +30 \beta ^2 \mu ^4+10 \mu ^6\right)+9 (\beta ^8+14 \beta ^6 \mu ^2+70 \beta ^4 \mu ^4\\
&+140 \beta ^2 \mu ^6+70 \mu ^8))\\
&\\
&\ga_9 :=-\beta  (9 \alpha ^8-42 \alpha ^6 \left(2 \beta ^2+3 \mu ^2\right)+126 \alpha ^4 \left(\beta ^4+5 \beta ^2 \mu ^2+5 \mu ^4\right)\\
&-18 \alpha ^2 \left(2 \beta ^6+21 \beta ^4 \mu ^2+70 \beta ^2 \mu ^4+70 \mu ^6\right)+\beta ^8+18 \beta ^6 \mu ^2+126 \beta ^4 \mu ^4\\
&+420 \beta ^2 \mu ^6+630 \mu ^8),\\
\end{aligned}\ee
\noindent
and
\be\label{speeds11ga}
\begin{aligned}
&\delta_{11} := \alpha  (\alpha ^{10}-11 \alpha ^8 \left(5 \beta ^2+2 \mu ^2\right)+66 \alpha ^6 (5 \beta ^4+12 \beta ^2 \mu ^2+3 \mu ^4)\\
&-462 \alpha ^4 \left(\beta ^6+6 \beta ^4 \mu ^2+9 \beta ^2 \mu ^4+2 \mu ^6\right)+33 \alpha ^2 (5 \beta ^8+56 \beta ^6 \mu ^2+210 \beta ^4 \mu ^4\\
&+280 \beta ^2 \mu ^6+70 \mu ^8)-11 (\beta ^{10}+18 \beta ^8 \mu ^2+126 \beta ^6 \mu ^4+420 \beta ^4 \mu ^6+630 \beta ^2 \mu ^8\\
&+252 \mu ^{10}))\\
&\\
&\ga_{11} :=-\beta  (-11 \alpha ^{10}+33 \alpha ^8 \left(5 \beta ^2+6 \mu ^2\right)-462 \alpha ^6 (\beta ^4+4 \beta ^2 \mu ^2+3 \mu ^4)\\
&+66 \alpha ^4 \left(5 \beta ^6+42 \beta ^4 \mu ^2+105 \beta ^2 \mu ^4+70 \mu ^6\right)-11 \alpha ^2 (5 \beta ^8+72 \beta ^6 \mu ^2+378 \beta ^4 \mu ^4\\
&+840 \beta ^2 \mu ^6+630 \mu ^8)+\beta ^{10}+22 \beta ^8 \mu ^2+198 \beta ^6 \mu ^4+924 \beta ^4 \mu ^6+2310 \beta ^2 \mu ^8\\
&+2772 \mu ^{10}).\\
\end{aligned}\ee
\end{defn}


\begin{thebibliography}{99}
\small{
%


\bibitem{AS} M. Abramowitz and  I. A. Stegun, \emph{Handbook of Mathematical Functions with Formulas, Graphs, and Mathematical Tables}, New York. Dover, 1972.




\bibitem{AleCar} M.A. Alejo and E. Cardoso, \emph{Nonlinear Stability of higher order mKdV breathers}, submitted, arxiv: 1804.02285.

\bibitem{AleCar2} M.A. Alejo and E. Cardoso, \emph{On the ill-posedness of the 5th-order Gardner equation}, submitted arXiv:1810.10434.

\bibitem{Ale1} M.A. Alejo, \emph{On the ill-posedness of the Gardner equation},  J. Math. Anal. Appl., {\bf396} no. 1,  256-260  (2012).


\bibitem{AleCk} M.A. Alejo and C. Kwak, \emph{ Global solutions and stability properties of the 5th order Gardner equation}, submitted, arXiv:1901.03350.

\bibitem{AM} M.A. Alejo  and  C. Mu\~noz, \emph{Nonlinear stability of mKdV breathers}, Comm. Math. Phys., \textbf{37} (2013), 2050--2080.

\bibitem{AM1} M.A. Alejo  and  C. Mu\~noz, \emph{Dynamics of complex-valued modified KdV solitons with applications to the stability of breathers},
Analysis and PDE 8-3 (2015) 629--674.


\bibitem{AM11} M.A. Alejo, \emph{Nonlinear stability of Gardner breathers}, Jour. Diff. Equat. \textbf{264}, n.2, 1192-1230 (2018).


\bibitem{AMF} M.A. Alejo, L. Fanelli   and C. Mu\~noz, {\it Stability and instability of breathers 
in the $U(1)$ Sasa-Satusuma and Nonlinear Schr\"odinger models}, preprint 2019.

\bibitem{AMP1} M.A. Alejo,  C. Mu\~noz and J.M. Palacios, \emph{On the variational structure of breather solutions I: the Sine-Gordon case}, Jour. Math. Anal. Appl. Vol.453/2 (2017) pp. 1111-1138.

\bibitem{AMP2} M.A. Alejo,  C. Mu\~noz and J.M. Palacios, \emph{On the variational structure of breather solutions II: periodic mKdV case},
 Elect. J. Diff. Eq., Vol. 2017 (2017), No. 56, pp. 1-26.










\bibitem{By}  P.F. Byrd and M.D. Friedman, \emph{Handbook of Elliptic Integrals for Engineers and Scientists}, 2nd ed., Springer-Verlag, New York and Heidelberg, 1971.



\bibitem{jesus1} A. \'Alvarez, F. Romero, J. Cuevas and JFR. Archilla, \emph{Moving breather collisions in Klein-Gordon chains of oscillators}, The European Physical Journal B 70 (2).
\bibitem{jesus2} J. Cuevas, Q. Hoq, H. Susanto and PG. Kevrekidis, \emph{Interlaced solitons and vortices in coupled DNLS lattices}, Physica D 238 (2009) 2216.

\bibitem{jesus3} JFR. Archilla, J. Cuevas and FR. Romero, \emph {Effect of breather existence on reconstructive transformations in mica muscovite}, AIP Conference Proceedings 982 (2008), 788.
















\bibitem{gomes}  J.F. Gomes, G.S. Fran\c ca and A.H. Zimerman, \emph{Nonvanishing boundary condition for the mKdV
hierarchy and the Gardner equation} 2012 J. Phys. A: Math. Theor. 45 015207.











\bibitem{KPV2} C.E. Kenig, G. Ponce and L. Vega, \emph{On the ill-posedness of some canonical dispersive equations}, Duke Math. J. \textbf{106}, no. 3, 617--633  (2001).


\bibitem{KKS1} P.G. Kevrekidis, A. Khare, A. Saxena  and  G. Herring, \textit{On some classes of mKdV periodic solutions}, Journal of Physics A: Mathematical and General, {\bf37}, 10959-10965  (2004).

\bibitem{KKS2} P.G. Kevrekidis, A. Khare and A. Saxena, \textit{Breather lattice and its stabilization for the modified Korteweg-de Vries equation}, Phys.Rev. E, {\bf68}, 0477011-0477014 (2003).















\bibitem{Mat}  Y. Matsuno, \emph{Bilinear transformation Method}. Volume 174 1st. Edition, Academic Press 1984.

\bibitem{Mat1} Y. Matsuno, \emph{Bilinearization of Nonlinear Evolution Equations: Higher Order mKdV}, Jour. Phys. Soc. Japan, {\bf 49}, n.2 (1980).

\bibitem{Mat2} Y. Matsuno, \emph{Personal communication}.


\bibitem{Miura} R.M. Miura, {\it Korteweg-de Vries equation and generalizations. I. A remarkable explicit nonlinear transformation}, J. Math. Phys. 9, no. 8 (1968), 1202-1204.



\bibitem{mupo}  C. Mu\~noz and G. Ponce, \emph{Breathers and the dynamics of solutions to the KdV type equations}, arXiv:1803.05475, accepted in Comm. Math. Phys. 

\bibitem{olver} P.J. Olver, {\it Evolution equations possessing infinitely many symmetries}, J. Math. Phys. 18, 1212 (1977).
%











%
}
\end{thebibliography}
\end{document}